\documentclass[12pt]{article}

\usepackage{amssymb,amsmath,latexsym}
\usepackage{amsthm}
\usepackage{bm}
\usepackage{enumerate}
\usepackage[
  pdftex,
  pdfauthor = {Evan Donald and Jason Swanson},
  pdftitle = {
    Conditional distributions for the nested Dirichlet process via sequential
    imputation
  },
  hidelinks
]{hyperref}
\usepackage{graphicx}
\usepackage{subcaption}
\usepackage{longtable}

\title{
  Conditional distributions for the nested Dirichlet process via sequential
  imputation
}
\author{
  Evan Donald\\
  University of Central Florida
  \and
  Jason Swanson\\
  University of Central Florida
}
\date{}

\newtheorem{thm}{Theorem}[section]
\newtheorem{cor}[thm]{Corollary}
\newtheorem{prop}[thm]{Proposition}
\newtheorem{assum}[thm]{Assumption}
\theoremstyle{remark}
\newtheorem{rmk}[thm]{Remark}

\numberwithin{equation}{section}

\DeclareMathOperator{\Bet}{Beta}
\DeclareMathOperator{\Dir}{Dir}
\DeclareMathOperator{\Gam}{Gamma}
\DeclareMathOperator{\Gamer}{Gamer}
\DeclareMathOperator{\Var}{Var}

\def\l{\ell}
\def\ol{\overline}

\def\wh{\widehat}
\def\wt{\widetilde}

\def\al{\alpha}
\def\be{\beta}
\def\ga{\gamma}
\def\Ga{\Gamma}
\def\de{\delta}
\def\De{\Delta}
\def\ep{\varepsilon}
\def\th{\theta}
\def\vth{\vartheta}
\def\ka{\kappa}
\def\la{\lambda}
\def\vpi{\varpi}
\def\vrho{\varrho}
\def\si{\sigma}
\def\ph{\varphi}
\def\Om{\Omega}

\def\bmeta{\bm{\eta}}
\def\bmmu{\bm{\mu}}
\def\bmnu{\bm{\nu}}
\def\bmth{\bm{\th}}
\def\bmp{\bm{p}}
\def\bmu{\bm u}
\def\bmX{{\bm X}}
\def\bmx{{\bm x}}
\def\bmY{{\bm Y}}
\def\bmy{{\bm y}}
\def\bmZ{{\bm Z}}
\def\bmz{{\bm z}}

\def\bN{\mathbb{N}}
\def\bR{\mathbb{R}}

\def\cD{\mathcal{D}}
\def\cF{\mathcal{F}}
\def\cG{\mathcal{G}}
\def\cL{\mathcal{L}}
\def\cS{\mathcal{S}}
\def\cT{\mathcal{T}}

\def\fm{\mathfrak{m}}
\def\fn{\mathfrak{n}}

\begin{document}

\maketitle

\begin{abstract}

We consider an array of random variables, taking values in a complete and
separable metric space, that exhibits a kind of symmetry which we call row
exchangeability. Given such an array, a natural model for Bayesian nonparametric
inference is the nested Dirichlet process (NDP). Exactly determining posterior
distributions for the NDP is infeasible, since the computations involved grow
exponentially with the sample size. In this paper, we present a new approach to
determining these posterior distributions that involves the use of sequential
imputation.

\bigskip

\noindent{\bf AMS subject classifications:} Primary 62G05;
secondary 60G57, 62D10, 62M20

\smallskip

\noindent{\bf Keywords and phrases:} exchangeability, Dirichlet processes,
Bayesian inference, importance sampling, sequential imputation

\end{abstract}

\section{Introduction}\label{S:intro}

\subsection{Motivation and main objective}

Consider a situation in which there are several agents, all from the same
population. Each agent undertakes a sequence of actions. These actions are
chosen according to the agent's particular tendencies. Although different agents
have different tendencies, there may be patterns in the population.

We observe a certain set of agents over a certain amount of time. Based on these
observations, we want to make probabilistic forecasts about two things:
\begin{itemize}
  \item The future behavior of the agents we have observed.
  \item The behavior of a new (unobserved) agent from the population.
\end{itemize}
Let $S$ denote the set of possible actions the agents may take. We assume that
$S$ is a complete and separable metric space. Let $\xi_{ij}$ denote the $j$th
action by the $i$th agent. We assume that $\{\xi_{\si(i),\tau_i(j)}\}$ and
$\{\xi_{ij}\}$ have the same finite-dimensional distributions whenever $\si$ and
$\tau_i$ are permutations. We will say that such an array is \emph{row
exchangeable}.

Note that if $\xi = \{\xi_{ij}\}$ is row exchangeable, then for each $i$, the
sequence $\xi_i = \{\xi_{ij}: j \in \bN\}$ is an exchangeable sequence of
$S$-valued random variables, and that the sequence of sequences, $\xi = \{\xi_i:
i \in \bN\}$ is also exchangeable. Let $M_1(S)$ denote the set of probability
measures on $S$. We equip $M_1(S)$ with the Prohorov metric, so that $M_1(S)$ is
also a complete and separable metric space. By de Finetti's theorem, there
exists a sequence of random Borel probability measures $\mu_1, \mu_2, \ldots$ on
$S$ and a random Borel probability measure $\vpi$ on $M_1(S)$ such that
\begin{enumerate}[(i)]
  \item given $\vpi$, the sequence $\mu_1, \mu_2, \ldots$ is i.i.d.\@ with
        distribution $\vpi$, and
  \item for each $i$, given $\mu_i$, the sequence $\xi_{i1}, \xi_{i2}, \ldots$
        is i.i.d.\@ with distribution $\mu_i$.
\end{enumerate}
We call the $\mu_i$ the \emph{row distributions} of the random array $\xi =
\{\xi_{ij}\}$, and we call $\vpi$ the \emph{row distribution generator}.

Our goal is to make inferences about the array $\xi$ based on observations of
some of its entries. We wish for this inference to be both Bayesian and
nonparametric. To facilitate Bayesian inference, we must place a prior
distribution on the random measure $\vpi$. We make the nonparametric choice of
letting $\vpi$ be a Dirichlet process. That is, $\vpi \sim \cD(\ka \rho)$ for
some $\ka > 0$ and some Borel probability measure $\rho$ on $M_1(S)$. To choose
the measure $\rho$, we first observe that
\begin{equation*}
  P(\mu_i \in B) = E[P(\mu_i \in B \mid \vpi)] = E[\vpi(B)] = \rho(B).
\end{equation*}
Hence, the measure $\rho$ is the prior distribution for $\mu_i$. In keeping with
our aim of nonparametric inference, we also let $\rho$ be the distribution of a
Dirichlet process. That is, $\rho = \cD(\ep \vrho)$ for some $\ep > 0$ and some
Borel probability measure $\vrho$ on $S$.

This gives us the model $\vpi \sim \cD(\ka \cD(\ep \vrho))$. In other words, the
process $\vpi$ is what Rodr\'iguez, Dunson, and Gelfand call a nested Dirichlet
process (NDP) in \cite{Rodriguez2008}. In that paper, the authors use a
motivating example in which the agents are different medical centers, and the
actions are the individual patient outcomes produced by these centers.

We call $\ka$ and $\ep$ the column and row concentrations of $\vpi$,
respectively, and we call $\vrho$ the base measure of $\vpi$. We adopt the
following notation:
\begin{equation*}
  X_{in} = \begin{pmatrix}
    \xi_{i1} & \cdots & \xi_{in}
  \end{pmatrix}, \qquad
  \bmX_{mn} = \begin{pmatrix}
    X_{1n} \\ \vdots \\ X_{mn}
  \end{pmatrix}
  = \begin{pmatrix}
    \xi_{11} & \cdots & \xi_{1n} \\
    \vdots & \ddots & \vdots \\
    \xi_{m1} & \cdots & \xi_{mm}
  \end{pmatrix}, \qquad
  \bmmu_m = \begin{pmatrix}
    \mu_1 \\ \vdots \\ \mu_m
  \end{pmatrix}.
\end{equation*}
Our objective is to make inferences about the future of the process $\xi$ based
on past observations. That is, if $M, N, N' \in \bN$ and $N < N'$, then we wish
to compute
\begin{equation}\label{main-goal}
  \cL(\xi_{ij}: i \le M, \; N < j \le N' \mid \bmX_{MN}).
\end{equation}
Here, the notation $\cL(X \mid Y)$ denotes the regular conditional distribution
of $X$ given $Y$.

As demonstrated in \cite{Rodriguez2008}, algorithms based on P\'olya urns are
infeasible in this situation. They require evaluating distributions where the
number of terms grows exponentially with the sample size. For the same reason,
the exact computation of \eqref{main-goal} is infeasible. (The exact computation
in the simplest nontrivial case, $M = 2$ and $S = \{0, 1\}$, is given in
\cite{Berry1979} and takes up a full page.) In \cite{Rodriguez2008}, the authors
deal with the infeasibility of posterior computation by using truncation. Here,
we take a different approach, motivated by the work of Liu and coauthors in
\cite{Kong1994,Liu1996}. In \cite{Liu1996}, Liu considered what is effectively
an NDP with $S = \{0, 1\}$. Using the method of sequential imputation, developed
in \cite{Kong1994}, Liu was able to determine the posterior distributions 
\eqref{main-goal} without truncation. Unfortunately, Liu's proof contains a
fatal flaw. In this paper, we correct that flaw, then generalize the method to
arbitrary $S$.

\subsection{The role of sequential imputation}

To explain how sequential imputation enters into the determination of
\eqref{main-goal}, consider the following. It is straightforward to verify that
\begin{equation*}
  P\left(
    \bigcap_{i = 1}^M \bigcap_{j = N_i + 1}^{O_i} \{\xi_{ij} \in A_{ij}\}
      \; \middle| \; \cG
  \right)
    = E\left[
      \prod_{i = 1}^M \prod_{j = N_i + 1}^{O_i} \mu_i(A_{ij})
        \; \middle| \; \cG
    \right],
\end{equation*}
for all Borel sets $A_{ij} \subseteq S$ and any sub-$\si$-algebra $\cG \subseteq
\si(X_{1 N_1}, X_{2 N_2}, \ldots, X_{M N_M})$. Taking $N_i = N$, $O_i = N'$, and
$\cG = \si(\bmX_{MN})$, we see that \eqref{main-goal} is entirely determined by
$\cL(\bmmu_M \mid \bmX_{MN})$. Note that
\[
  \cL(\bmmu_M \mid \bmX_{MN}; d\nu)
    = \cL(
      \mu_2, \ldots, \mu_M \mid \bmX_{MN}, \mu_1 = \nu_1;
      d\nu_2 \cdots d\nu_M
    ) \, \cL(\mu_1 \mid \bmX_{MN}; d\nu_1).
\]
Iterating this, we see that we can determine $\cL(\bmmu_M \mid \bmX_{MN})$ if we
know
\begin{equation}\label{main-goal2}
  \cL(\mu_m \mid \bmX_{MN}, \bmmu_{m - 1}), \quad 1 \le m \le M.
\end{equation}
In general, $\bmX_{mN}$ and $\{ (\mu_j, X_{jN})\}_{j = m + 1}^M$ are
conditionally independent given $\bmmu_m$. Hence, in \eqref{main-goal2}, the
first $m - 1$ rows of $\bmX_{MN}$ can be omitted. In other words, the posterior
distribution \eqref{main-goal} is entirely determined by the conditional
distributions,
\begin{equation}\label{main-goal3}
  \cL(\mu_m \mid X_{mN}, \ldots, X_{MN}, \bmmu_{m - 1}),
\end{equation}
where $1 \le m \le M$.

When $m < M$, the determination of \eqref{main-goal3} involves conditioning on
more than one row of the array $\xi$. This is exactly the issue that makes
P\'olya-urn-based algorithms and exact computations of \eqref{main-goal}
infeasible. The number of required calculations grows exponentially with the
number of rows. If, instead of \eqref{main-goal3}, we wanted to compute
\begin{equation}\label{1rowConDis}
  \cL(\mu_m \mid X_{mN}, \bmmu_{m - 1}),
\end{equation}
for $1 \le m \le M$, then this is feasible. The problem, then, is to find a way
to use \eqref{1rowConDis} to determine \eqref{main-goal}.

This is where sequential imputation is used. The entire row of data, $\xi_m = 
\{\xi_{mj}\}_{j = 1}^\infty$, is enough to determine $\mu_m$. But in 
\eqref{1rowConDis}, we observe $X_{mN}$, which is only part of this row of data.
Hence, we are faced with a sequence, indexed by $m$, of missing data problems.
This is precisely the situation that sequential imputation is designed to
handle.

\subsection{Other related models}

\subsubsection{The dependent Dirichlet process}

Although the NDP arises naturally when considering row exchangeable arrays, it
is just one of many models used for nonparametric Bayesian inference. In the
NDP, we see that we have a sequence $\mu = \{\mu_i\}_{i = 1}^\infty$ of
dependent Dirichlet processes. The study of dependent Dirichlet processes goes
back at least to \cite{MacEachern1999, MacEachern2000}. Since a Dirichlet
process is almost surely a discrete measure, it is characterized by its atoms
(the countable set of points on which it is supported) and its weights (the
amount of mass it assigns to each atom). In \cite{MacEachern1999,
MacEachern2000}, a sequence of dependent Dirichlet processes is defined by a
specific construction of the weights and atoms. The resulting process has come
to be known by the name, dependent Dirichlet process (DDP). It should be noted,
though, that not every family of dependent Dirichlet processes is a DDP. Our
sequence $\mu$, for instance, is not a DDP, but is rather a variation of the
DDP. In \cite{Barrientos2012}, an alternative, equivalent definition of the DDP
was given in terms of copulas. For a survey of the DDP and related models, see
\cite{Quintana2022}.

Two common categories of DDPs are the single-weights DDPs and the single-atoms
DDPs. In the former, all the Dirichlet processes in the DDP share the same
weights; in the latter, they share the same atoms. The NDP does not fit into
either of these categories. In fact, as long as $\vrho$ is non-atomic, either
$\mu_i = \mu_j$ or $\mu_i$ and $\mu_j$ have no atoms and no weights in common.

\subsubsection{Mixture models}

There are a number of popular variations of the DDP. For instance, in
\cite{Dunson2007}, the authors propose a model which begins with a family of
independent Dirichlet processes. The dependent Dirichlet processes are then
constructed as weighted mixtures of the independent ones, with the dependence
structure determined by the weights. In a slightly different approach,
\cite{Mueller2004} proposes a model in which the dependence is created by a
common underlying Dirichlet process. That is, $\mu_j = c \wt \mu_0 + (1 - c) \wt
\mu_j$, where $\{\wt \mu_j\}_{j = 0}^\infty$ are independent Dirichlet
processes. The parameter $c$ allows for control over the degree of dependence
among the different $\mu_j$. The authors call this a hierarchical Dirichlet
process mixture model.

The hierarchical Dirichlet process mixture of \cite{Mueller2004} should not be
confused with the hierarchical Dirichlet process (HDP), which was introduced in
\cite{Teh2006}. Compared to the mixture processes of \cite{Dunson2007} and
\cite{Mueller2004}, the HDP seems, at least on the surface, to be much more
closely related to the NDP. In fact, though, they are quite different. The HDP
is a Dirichlet process whose base measure is a Dirichlet process. The NDP,
however, is a Dirichlet process whose base measure is \emph{the law of} a
Dirichlet process. If we are not careful about distinguishing between a process
and its law, then we could easily mistake one for the other.

The HDP and NDP, in fact, have different state spaces. The HDP takes values in
$M_1(S)$, whereas the NDP takes values in $M_1(M_1(S))$. For example, if we take
$\ka'$ to be random and $\rho' \sim \cD(\ep \vrho)$, then we can define $\la$ to
be an HDP by using the conditional distribution ${\la \mid \ka', \rho'} \sim
\cD(\ka' \rho')$. Note that $\rho'$ is an $M_1(S)$-valued random variable. In
contrast, for the NDP, we take $\ka$ and $\rho$ to be nonrandom and $\rho =
\cD(\ep \vrho)$. In this case, $\rho$ is a nonrandom element of $M_1(M_1(S))$.
We then define $\vpi$ to be an NDP by the unconditional distribution $\vpi \sim
\cD(\ka \rho)$.

\subsubsection{The infinite relational model}

A model that does bear a close connection to the NDP is the infinite relational
model (IRM), introduced independently in both \cite{Kemp2006} and \cite{Xu2006}.
It is a cluster-based model that can be regarded as a kind of nonparametric
stochastic block model, and is common in the machine learning literature. For a
survey of the IRM and other Bayesian models of exchangeable structures, see
\cite{Orbanz2015}.

We could obtain an IRM from our process $\xi$ by only a slight modification. To
understand this modification, we should clarify that although (i) and (ii) above
are consequences of the row exchangeability of $\xi$, they do not characterize
row exchangeability. For example, an array $\xi$ is said to be separately
exchangeable if $\{\xi_{\si(i), \tau(j)}\}$ and $\{\xi_{ij}\}$ have the same
finite-dimensional distributions whenever $\si$ and $\tau$ are permutations. A
separately exchangeable array also satisfies (i) and (ii). But a row
exchangeable array satisfies
\begin{enumerate}[(i)]
  \setcounter{enumi}{2}
  \item the entries of the array $\{\xi_{ij}\}$ are conditionally i.i.d.\@ given
        $\mu$,
\end{enumerate}
whereas a separately exchangeable array does not. Because of (iii), our process
satisfies $\cL(\bmX_{mn} \mid \mu) = \prod_{i = 1}^m \mu_i^n$, a fact that we
will use again in Section \ref{S:cond1row}. In particular, we have
$\cL(\xi_{11}, \xi_{21} \mid \mu_1, \mu_2) = \mu_1 \times \mu_2$, and this is
true regardless of whether $\mu_1 = \mu_2$ or not.

Now, suppose we modify our process so that whenever $\mu_i = \mu_{i'}$, the rows
$\xi_i$ and $\xi_{i'}$ are identical. That is, $\mu_i = \mu_{i'}$ implies
$\xi_{ij} = \xi_{i'j}$ for all $j$. In this case, the array $\xi$ no longer
satisfies (iii) and is no longer row exchangeable. It is, however, separately
exchangeable. In fact, the array $\xi$, modified in this way, would be an
instance of an IRM called a simple IRM. A general IRM is obtained from a simple
IRM through a process called randomization.

We can apply randomization to any random array, so let use drop any specific
assumptions for now and just let $\xi$ be an arbitrary random array of elements
of $S$. To obtain a randomization of $\xi$, let $(T, \cT)$ be a measurable space
and let $Q$ be a probability kernel from $S$ to $T$. Define $\Xi = \{\Xi_{ij}\}$
so that $\{\Xi_{ij}\}$ is conditionally i.i.d.\@ given $\xi$ and $\Xi_{ij} \mid
\xi \sim Q(\xi_{ij})$. Then the array $\Xi$ is called a $Q$-randomization of
$\xi$. If we apply $Q$-randomization to the simple IRM in the previous
paragraph, then we obtain a general IRM.

The IRM is a model that is used when the columns of $\xi$ correspond to
fundamentally distinct objects that remain fixed from row to row. For instance,
suppose the agents are movie critics, and their $j$th action is to review the
$j$th movie on some fixed list of movies that is shared by all critics. Suppose
we knew the exact reviewing tendencies of the first two critics. That is, we
know $\mu_1$ and $\mu_2$. Even then, if we observe $\xi_{11}$, the first
critic's review of the first movie, then we would learn something about the
first movie. This could potentially affect our probabilities for $\xi_{21}$, the
second critic's review of the first movie. In other words, $\xi_{11}$ and
$\xi_{21}$ would not be conditionally independent given $\mu_1$ and $\mu_2$.
This is distinctly different from the motivating example in
\cite{Rodriguez2008}, where the agents are different medical centers, and the
actions are the individual patient outcomes produced by those centers. In the
movie critic scenario, an NDP is not an appropriate model, whereas an IRM would
be reasonable. In the language of the machine learning literature, the rows,
which represents the agents, exhibit clustering in both scenarios. But only in
the movie critic scenario do the columns exhibit clustering.

\subsection{Outline of paper}

In Section \ref{S:backgrnd}, we give some necessary background information and
establish the notational conventions that we will use throughout the paper. In
Section \ref{S:seqImp}, we give a precise formulation of the method of
sequential imputation as it appears in \cite{Kong1994}. This formulation is
given in Theorem \ref{T:Kong-Liu}. Liu's error in \cite{Liu1996} is to apply
this result to the simple NDP with $S = \{0, 1\}$ despite the fact that the
hypotheses do not hold. To correct this error, we give a new proof under weaker
hypotheses, and present this in Theorem \ref{T:seqImpGen}.

In Section \ref{S:seqImpNDP}, we apply Theorem \ref{T:seqImpGen} to the NDP with
a general state space $S$. Our main results are Theorem \ref{T:seqImpNDP} and
Corollary \ref{C:seqImpNDP}. Finally, in Section \ref{S:examples}, we present
several hypothetical examples to illustrate the use of our main results. See
\url{https://github.com/jason-swanson/ndp} for the code used to generate the
simulations in Section \ref{S:examples}.

\section{Notation and background}\label{S:backgrnd}

\subsection{General notation}

Throughout the paper, we fix a complete and separable metric space $S$. When
needed, we let $\cS$ denote its Borel $\si$-algebra. As noted in the
introduction, we write $M_1 = M_1(S)$ for the set of Borel probability measures
on $S$, and equip $M_1$ with the Prohorov metric so that $M_1$ is itself a
complete and separable metric space.

Let $(T, \cT)$ be a measurable space. Recall that a probability kernel from $T$
to $S$ is a measurable function $\mu: T \to M_1(S)$. If $\mu$ is any function
from $T$ to $M_1(S)$, measurable or not, we write $\mu(t, B)$ for $(\mu(t))(B)$.
Such a function is a kernel if and only if $\mu(\cdot, B)$ is measurable for
each Borel set $B$. Note that a random probability measure on $S$ is a
probability kernel from $\Om$ to $S$. Also, if $\mu$ is a probability kernel
from $T$ to $S$ and $Y$ is a $T$-valued random variables, then $\mu(Y)$ is a
random probability measure on $S$.

Now let $S'$ be another complete and separable metric space. Let $\ga$ be a
probability kernel from $T$ to $S$ and $\ga'$ a probability kernel from $T
\times S$ to $S'$. (We allow the possibility that $T$ is a singleton, in which
case $\ga$ is a probability measure on $S$ and $\ga'$ is a probability kernel
from $S$ to $S'$.) We write $\ga \ga'$ to denote the probability kernel from $T$
to $S \times S'$ characterized by
\begin{equation*}
  (\ga \ga')(y, A \times A') = \int_A \ga'(y, z, A') \, \ga(y, dz).
\end{equation*}
In particular, this means
\[
  \int_{S \times S'} f(z, z') \, (\ga \ga')(y, dz \, dz')
    = \int_S \int_{S'} f(z, z') \, \ga'(y, z, dz') \, \ga(y, dz).
\]
As shorthand for this equation, we write
\[
  (\ga \ga')(y, dz \, dz') = \ga'(y, z, dz') \, \ga(y, dz).
\]
If $T$ is a singleton, then $\ga \ga'$ is a probability measure and $(\ga \ga')
(dz \, dz') = \ga'(z, dz') \, \ga(dz)$.

We write $\cL(X)$ and $\cL(X \mid Y)$ for the distribution of $X$ and the
regular conditional distribution of $X$ given $Y$, respectively. We use
semicolons to indicate evaluation, so that $\cL(X; B) = (\cL(X))(B)$ and $\cL(X
\mid Y; B) = P(X \in B \mid Y)$. We also adopt the usual notation, $X \sim \mu$
and $X \mid Y \sim \mu$, to mean $\cL(X) = \mu$ and $\cL(X \mid Y) = \mu(Y)$,
respectively. In the case $\cL(X \mid Y) = \mu(Y)$, the probability kernel $\mu$
is only determined $\mu(Y)$-a.e. Nonetheless, if a particular $\mu$ has been
fixed, we use the notation $\cL(X \mid Y = y)$ to denote the probability measure
$\mu(y, \cdot)$.

\subsection{Dirichlet processes}

Given a nonzero, finite measure $\al$ on $S$, a Dirichlet process on $S$ with
parameter $\al$ is a random probability measure $\la$ on $S$ that satisfies
\begin{equation}\label{DirProcFDD}
  \cL(\la(B_0), \ldots, \la(B_d)) = \Dir(\al(B_0), \ldots, \al(B_d)),
\end{equation}
whenever $\{B_0, \ldots, B_d\} \subseteq \cS$ is a partition of $S$. The
right-hand side of \eqref{DirProcFDD} is the Dirichlet distribution on the
simplex $\De^d$. We write $\cD(\al)$ to denote the law of a Dirichlet process
with parameter $\al$. Since a Dirichlet process is an $M_1$-valued random
variable, it follows that $\cD(\al)$ is a Borel probability measure on $M_1$.
That is, $\cD(\al) \in M_1(M_1)$. Given a Borel set $B \subseteq M_1$, we write
$\cD(\al, B)$ for $(\cD(\al))(B)$.

With $\al$ as above, let $\ka = \al(S) > 0$ and $\rho = \ka^{-1} \al$, so that
$\rho \in M_1$. We typically write $\cD(\al) = \cD(\ka \rho)$, and think of the
law of a Dirichlet process as being determined by two parameters, a positive
number $\ka \in (0, \infty)$ and a probability measure $\rho \in M_1$. We call
the measure $\rho$ the base measure, or base distribution, and the number $\ka$
the concentration parameter. If $\ph \in L^1(\rho)$, then
\begin{equation*}
  E \int \ph \, d\la = \int \ph \, d\rho.
\end{equation*}
Taking $\ph = 1_A$ gives the special case, $E[\la(A)] = \rho(A)$. This and other
basic properties of the Dirichlet process can be found in \cite{Ferguson1973}.

A sequence of samples from a Dirichlet process $\la \sim \cD(\ka \rho)$ is a
sequence $\eta = \{\eta_i\}_{i = 1}^\infty$ that satisfies ${\eta \mid \la} \sim
\la^\infty$. We adopt the notation $\bmeta_n = (\eta_1, \ldots, \eta_n)$ and
$\bmx_n = (x_1, \ldots, x_n) \in S^n$. Note that for fixed $i$, we have
\begin{equation*}
  P(\eta_i \in A) = E[P(\eta_i \in A \mid \la)] = E[\la(A)] = \rho(A).
\end{equation*}
Thus, $\rho$ represents our prior distribution on the individual $\eta_i$, in
the case that we have not observed any of their values. As shown in
\cite[Theorem 3.1]{Ferguson1973}, the posterior distribution is given by
\begin{align}
  \cL(\la \mid \bmeta_n)
    &= \cD\bigg(
      \al + \sum_{i = 1}^n \de_{\eta_i}
    \bigg), \text{ and}\label{Dir-Bayes}\\
  \cL(\eta_{n + 1} \mid \bmeta_n)
    &= \frac \ka {\ka + n} \rho + \frac n {\ka + n} \wh \rho_n,
    \label{Dir-post}
\end{align}
where $\wh \rho_n = n^{-1} \sum_{i = 1}^n \de_{\eta_i}$ is the empirical
distribution of $\bmeta_n$.

The next proposition expresses \eqref{Dir-Bayes} in a purely analytic form.

\begin{prop}
  Let $\al$ be a nonzero, finite measure on $S$. Then
  \begin{equation}\label{DirBayesIn}
    \int_B \nu^n(A) \, \cD(\al, d\nu) = \int_{M_1} \int_A {
      \cD\bigg(\al + \sum_{i = 1}^n \de_{x_i}, B \bigg)
    } \, \nu^n(d\bmx_n) \, \cD(\al, d\nu),
  \end{equation}
  for every $A \in \cS^n$ and every Borel set $B \subseteq M_1$.
\end{prop}

\begin{proof}
  We first note that
  \[
    P(\la \in B, \bmeta_n \in A) = E[1_B(\la) P(\bmeta_n \in A \mid \la)]
      = E[1_B(\la) \la^n(A)]
      = \int_B \nu^n(A) \, \cD(\al, d\nu).
  \]
  On the other hand, by \eqref{Dir-Bayes}, we have
  \begin{align*}
    P(\la \in B, \bmeta_n \in A)
      &= E[1_A(\bmeta_n) P(\la \in B \mid \bmeta_n)]\\
    &= E\bigg[
      1_A(\bmeta_n) \cD\bigg( \al + \sum_{i = 1}^n \de_{\eta_i}, B \bigg)
    \bigg]\\
    &= E\bigg[ E\bigg[
      1_A(\bmeta_n) \cD\bigg( \al + \sum_{i = 1}^n \de_{\eta_i}, B \bigg)
    \;\bigg|\; \la \bigg] \bigg]\\
    &= E\bigg[ \int_A {
      \cD\bigg( \al + \sum_{i = 1}^n \de_{x_i}, B \bigg)
    } \, \la^n(d\bmx_n) \bigg]\\
    &= \int_{M_1} \int_A {
      \cD\bigg(\al + \sum_{i = 1}^n \de_{x_i}, B \bigg)
    } \, \nu^n(d\bmx_n) \, \cD(\al, d\nu),
  \end{align*}
  which proves \eqref{DirBayesIn}.
\end{proof}

We can also use \eqref{Dir-post} to obtain a recursive formula for the
distribution of $\bmeta_n$. Let $\rho_n = \cL(\bmeta_n)$. Suppose $f: S^{n + 1}
\to \bR$ is bounded and measurable. Then
\[
  \int_{S^{n + 1}} f \, d\rho_{n + 1} = E[f(\bmeta_{n + 1})]
    = E[E[f(\bmeta_{n + 1}) \mid \bmeta_n]]
    = E\bigg[\int_S {
      f(\bmeta_n, x_{n + 1})
    } \, \cL(\eta_{n + 1} \mid \bmeta_n; dx_{n + 1})\bigg].
\]
Using \eqref{Dir-post}, this gives
\[
  \int_{S^{n + 1}} f \, d\rho_{n + 1}
    = \frac \ka {\ka + n} E\bigg[\int_S {
      f(\bmeta_n, x_{n + 1})
    } \, \rho(dx_{n + 1})\bigg]
    + \frac 1 {\ka + n} \sum_{i = 1}^n E[f(\bmeta_n, \eta_i)].
\]
Hence,
\begin{multline}\label{jntLawSamp}
  \int_{S^{n + 1}} f \, d\rho_{n + 1} = \frac \ka {\ka + n} \, {
    \int_{S^n} \int_S f(\bmx_{n + 1}) \, \rho(dx_{n + 1}) \, \rho_n(d\bmx_n)
  }\\
  + \frac 1 {\ka + n} \, {
    \sum_{i = 1}^n \int_{S^n} f(\bmx_n, x_i) \, \rho_n(d\bmx_n)
  },
\end{multline}
for all $n \in \bN$.

\subsection{Mixtures of Dirichlet processes}

Now let $\al$ be a random measure on $S$ such that $\al(S) \in (0, \infty)$ a.s.
Let $\la$ be a random probability measure on $S$ such that $\la \mid \al \sim
\cD(\al)$. In this case, $\la$ is called a mixture of Dirichlet processes on $S$
with mixing distribution $\cL(\al)$. We also let $\ka = \al(S)$ and $\rho =
\al/\al (S)$, so that $\ka$ and $\rho$ are random variables taking values in
$(0, \infty)$ and $M_1$, respectively. Some fundamental properties of these
mixtures are given in \cite{Antoniak1974}.

A sequence of samples from $\la$ is a sequence $\eta = \{\eta_i\}_{i =
1}^\infty$ that satisfies ${\eta \mid \la, \al} \sim \la^\infty$. In this case, \eqref{Dir-Bayes} generalizes to
\begin{equation}\label{DirBayesMx}
  \cL(\la \mid \bmeta_n, \al)
    = \cD\bigg(
      \al + \sum_{i = 1}^n \de_{\eta_i}
    \bigg),
\end{equation}
for any $n \in \bN$.

Now let $(T, \cT)$ be a measurable space. Fix $n \in \bN$ and let $Y$ be a
$T$-valued random variable such that $Y$ and $(\la, \al)$ are conditionally
independent given $\bmeta_n$. This holds, for example, if $Y$ is a function of
$\bmeta_n$ and $W$, where $W$ is some noise that is independent of $(\la, \al)$.
In other words, we can think of $Y$ as a noisy observation of $\bmeta_n$.

The following result extends \eqref{Dir-Bayes} to noisy observations of data
generated by a Dirichlet mixture. A special case of this appears as \cite
[Theorem 3]{Antoniak1974}.

\begin{thm}
  With notation as above, we have
  \begin{equation}\label{DirNoisy1}
    \cL(\la \mid Y) = \int_{(0, \infty) \times M_1 \times S^n} {
      \cD\bigg(t \nu + \sum_{i = 1}^n \de_{x_i}\bigg)
    } \, \cL(\ka, \rho, \bmeta_n \mid Y; dt \, d\nu \, dx),
  \end{equation}
  In particular, if $\al$ is not random, as in
  \eqref{Dir-Bayes}, so that $Y$ and $\la$ are conditionally independent given
  $\bmeta_n$, then
  \begin{equation}\label{DirNoisy2}
    \cL(\la \mid Y) = \int_{S^n} {
      \cD\bigg(\al + \sum_{i = 1}^n \de_{x_i}\bigg)
    } \, \cL(\bmeta_n \mid Y; d\bmx_n)
  \end{equation}
\end{thm}

\begin{proof}
  Let $B \subseteq M_1$ be Borel measurable. Then
  \begin{equation}\label{DirNoisy3}
    P(\la \in B \mid Y) = E[P(\la \in B \mid Y, \bmeta_n) \mid Y]
      = E[P(\la \in B \mid \bmeta_n) \mid Y],
  \end{equation}
  since $Y$ and $\la$ are independent given $\bmeta_n$. Similarly,
  \begin{equation}\label{DirNoisy4}
    P(\la \in B \mid \bmeta_n)
      = E[P(\la \in B \mid \bmeta_n, \al) \mid \bmeta_n]
      = E[P(\la \in B \mid \bmeta_n, \al) \mid Y, \bmeta_n],
  \end{equation}
  since $Y$ and $\al$ are independent given $\bmeta_n$. Substituting 
  \eqref{DirNoisy4} in \eqref{DirNoisy3}, we have
  \[
    P(\la \in B \mid Y)
      = E[E[P(\la \in B \mid \bmeta_n, \al) \mid Y, \bmeta_n] \mid Y]
      = E[P(\la \in B \mid \bmeta_n, \al) \mid Y].
  \]
  By \eqref{DirBayesMx}, this gives
  \begin{align*}
    P(\la \in B \mid Y) &= E\bigg[
      \cD\bigg(\al + \sum_{i = 1}^n \de_{\eta_i}, B\bigg)
    \;\bigg|\; Y \bigg]\\
    &= \int_{(0, \infty) \times M_1 \times S^n} {
      \cD\bigg(t \nu + \sum_{i = 1}^n \de_{x_i}, B\bigg)
    } \, \cL(\ka, \rho, \bmeta_n \mid Y; dt \, d\nu \, dx),
  \end{align*}
  which is \eqref{DirNoisy1}. In the case that $\al$ is not random, this
  reduces to \eqref{DirNoisy2}.
\end{proof}

\section{Sequential imputation}\label{S:seqImp}

As discussed in Section \ref{S:intro}, we aim to find a way to use
\eqref{1rowConDis} to compute $\cL(\bmmu_M \mid \bmX_{MN})$. If we do this via
simulation, then we must find a way to use \eqref{1rowConDis} to simulate
$\bmmu_M$ according to the conditional distribution $\cL (\bmmu_M \mid
\bmX_{MN})$. One approach would be to simulate $\mu_1$ according to the
distribution $\cL(\mu_1 \mid X_{1N})$, and then use that simulated value to
simulate $\mu_2$ according to $\cL(\mu_2 \mid X_{2N}, \mu_1)$, and so on.
However, if we do that, then we would not be simulating $\bmmu_M$ according to
its correct conditional distribution, since our simulation of $\mu_m$ would not
take into account observations from higher-numbered rows.

One way to fix this is to do many such incorrect simulations of $\bmmu_M$. Let
$K$ be the number of incorrect simulations we generate. Some of these $K$
simulations will be ``more incorrect'' than others. We then assign the $K$
simulated values weights according to their level of correctness, and choose one
of them randomly, with probabilities proportional to those weights. If we assign
the weights appropriately, then the distribution of the chosen value will
converge to $\cL(\bmmu_M \mid \bmX_{MN})$ as $K$ tends to infinity.

This is the method of sequential imputation, first introduced in
\cite{Kong1994}. It is an application of the more general method of importance
sampling that originated in \cite{Kloek1978}. In Sections \ref{S:impSamp} and
\ref{S:ess}, we give a generalized formulation of importance sampling, along
with a discussion of effective sample size. In Section \ref{S:seqImpSim}, we
lay out the definitions and constructions that are needed in sequential
imputation. In Sections \ref{S:sim-dens} and \ref{S:pfDens}, we prove that
sequential imputation leads asymptotically to the desired conditional
expectation. The proof in Section \ref{S:pfDens} is a rigorous presentation of
the proof in \cite{Kong1994}.

Unfortunately, as we will see in Section \ref{S:seqImpNDP}, this version of
sequential imputation does not apply to the NDP. In Section \ref{S:pfNoDens},
therefore, we provide a new proof under more general assumptions. This new
result, given in Theorem \ref{T:seqImpGen}, will allow us in Section 
\ref{S:seqImpNDP} to apply sequential imputation to the NDP.

\subsection{Importance sampling}\label{S:impSamp}

Importance sampling is a method of approximating a particular probability
distribution using samples from a different distribution. The samples themselves
will vary in how ``important'' they are in determining the distribution of
interest. This is modeled by assigning different weights to the samples.

We begin by presenting, without commentary, the formal statement of the method
of importance sampling in Theorem \ref{T:imp-samp} below. We then describe in
Remark \ref{R:imp-samp} the intuitive interpretation of the method.

Let $(T, \cT)$ be a measurable space. Let $Z$ be an $S$-valued random variable
and $Y$ a $T$-valued random variable. Let $\fm^*$ be a probability kernel from
$T$ to $S$ such that $\cL(Z \mid Y) \ll \fm^*(Y)$ a.s. Assume there exist
measurable functions $w: T \times S \to \bR$ and $h: T \to [0, \infty)$ such
that $h(Y) > 0$ a.s., $E h(Y) < \infty$, and
\begin{equation}\label{impSampCon}
  w(Y, \cdot) = h(Y) \, \frac {d\cL(Z \mid Y)} {d\fm^*(Y)} \quad \text{a.s.}
\end{equation}
Define $Z^*$ so that $Z^* \mid Y \sim \fm^*(Y)$ and let $W = w(Y, Z^*)$. Let $
\{(Z^{*, k}, W_k)\}_{k = 1}^\infty$ be copies of $(Z^*, W)$ that are
i.i.d.\@ given $Y$. Define $\wt Z^K$ so that
\begin{equation}\label{impSampDef}
  \wt Z^K \mid Z^{*, 1}, \ldots, Z^{*, K}, Y \propto \sum_{k = 1}^K {
    W_k \de_{Z^{*, k}}
  }
\end{equation}

\begin{thm}\label{T:imp-samp}
  With the notation and assumptions given above, we have
  \begin{equation}\label{imp-samp}
    \cL(\wt Z^K \mid Z^{*, 1}, \ldots, Z^{*, K}, Y) \to \cL(Z \mid Y)
    \quad \text{a.s.}
  \end{equation}
  as $K \to \infty$.
\end{thm}

\begin{rmk}\label{R:imp-samp}
  The interpretation of Theorem \ref{T:imp-samp} is the following. We observe
  $Y$ and we wish to determine $\cL(Z \mid Y)$. Unfortunately, for one reason or
  another, this is not directly possible. Instead, we are only able to determine
  a different distribution, $\fm^*(Y)$, which we call the \emph{simulation
  measure}. Using $\fm^*(Y)$, we generate an i.i.d.\@ collection of samples,
  $Z^{*, 1}, \ldots, Z^{*, K}$. The $k$-th sample, $Z^{*, k}$, gets assigned the
  weight $W_k = w(Y, Z^{*, k})$, where $w$ is some function satisfying
  \eqref{impSampCon}. We then use these weights to randomly choose one of the
  $K$ samples. The randomly chosen sample is denoted by $\wt Z^K$. Theorem
  \ref{T:imp-samp} says that if $K$ is large, then the law of $\wt Z^K$ is close
  to $\cL(Z \mid Y)$.
\end{rmk}

\begin{proof}[Proof of Theorem \ref{T:imp-samp}]
  First note that
  \begin{equation}\label{imp-samp-2}
    \begin{split}
      E[w(Y, Z^*) \mid Y] &= \int_S w(Y, z) \, \fm^*(Y, dz)\\
        &= \int_S h(Y) \frac{d\cL(Z \mid Y)}{d\fm^*(Y)} \fm^*(Y, dz)\\
        &= h(Y) \int_S d\cL(Z \mid Y; dz)\\
        &= h(Y).
    \end{split}
  \end{equation}
  Hence, $E w(Y, Z^*) = E h(Y) < \infty$. Now let $f: S \to \bR$ be bounded
  and measurable. By the conditional law of large numbers, we have
  \[
    \frac 1 K \sum_{k = 1}^K {
      w(Y, Z^*) f(Z^{*, k})
    } \to E[w(Y, Z^*) f(Z^*) \mid Y] \quad \text{a.s.}
  \]
  But
  \begin{multline*}
    E[w(Y, Z^*) f(Z^*) \mid Y] = \int_S w(Y, z) f(z) \, \fm^*(Y, dz)\\
    = h(Y) \int_S f(z) \, \cL(Z \mid Y; dz) = h(Y) E[f(Z) \mid Y].
  \end{multline*}
  Therefore, since $h(Y) > 0$ a.s.,
  \begin{align*}
    E[f(\wt Z^K) \mid Z^{*, 1}, \ldots, Z^{*, K}, Y] &= \frac{
      \sum_{k = 1}^K w(Y, Z^{*, k}) f(Z^{*, k})
    }{
      \sum_{k = 1}^K w(Y, Z^{*, k})
    }\\
    &\to \frac{h(Y) E[f(Z) \mid Y]}{h(Y) E[1 \mid Y]}\\
    &= E[f(Z) \mid Y],
  \end{align*}
  and this proves \eqref{imp-samp}.
\end{proof}

\subsection{Effective sample size}\label{S:ess}

Let $f: S \to \bR$ be continuous and bounded. By \eqref{imp-samp}, if $K$ is
large, then
\[
  \frac{\sum_{k = 1}^K W_k f(Z^{*, k})}{\sum_{k = 1}^K W_k}
    \approx E[f(Z) \mid Y].
\]
On the other hand, by the conditional law of large numbers, if $\{Z^k\}_{k =
1}^\infty$ are copies of $Z$ that are i.i.d.\@ given $Y$, then
\[
  \frac 1 K \sum_{k = 1}^K f(Z^k) \approx E[f(Z) \mid Y].
\]
This latter estimate of $E[f(Z) \mid Y]$ is presumably more efficient, in the
sense that smaller $K$ values are needed. This is because in the latter
estimate, we are generating values directly from $\cL(Z \mid Y)$, rather than
from the modified distribution $\fm^*(Y)$.

In an effort to measure this difference in efficiency, let $K$ be a given number
of weighted samples. We wish to find a number $K_e$ such that
\[
  \Var\bigg(
    \frac{\sum_{k = 1}^K W_k f(Z^{*, k})}{\sum_{k = 1}^K W_k}
  \;\bigg|\; Y \bigg) \approx \Var\bigg(
    \frac 1 {K_e} \sum_{k = 1}^{K_e} f(Z^k)
  \;\bigg|\; Y \bigg).
\]
The right-hand side is $K_e^{-1} \Var(f(Z) \mid Y)$. In \cite{Kong1992}, it is
shown that
\[
  \Var\bigg(
    \frac{\sum_{k = 1}^K W_k f(Z^{*, k})}{\sum_{k = 1}^K W_k}
  \;\bigg|\; Y \bigg) \approx \frac 1 K \Var(f(Z) \mid Y)\bigg(
    1 + \Var\bigg(\frac W {h(Y)} \;\bigg|\; Y \bigg)
  \bigg).
\]
We therefore define
\[
  K_e = \frac K {1 + \Var\left(\frac W {h(Y)} \;\middle|\; Y\right)},
\]
which is called the effective sample size.

By \eqref{imp-samp-2}, we have
\[
  \Var\left(\frac W {h(Y)} \;\middle|\; Y\right)
    = \frac{\Var(W \mid Y)}{h(Y)^2}
    = \frac{\Var(W \mid Y)}{E[W \mid Y]^2}.
\]
Therefore,
\[
  K_e = K\left(1 + \frac{\Var(W \mid Y)}{E[W \mid Y]^2}\right)^{-1}.
\]
If we approximate $E[W \mid Y]$ by the sample mean, $\ol W = K^{-1} \sum_{k =
1}^K W_k$, and $\Var(W \mid Y)$ by the population variance, $\wt \si^2 = (K^ 
{-1} \sum_{k = 1}^K W_k^2) - \ol W^2$, then we have $K_e \approx K_e'$, where
\[
  K_e' = \frac K {1 + \wt \si^2/\ol W^2}
    = \frac{\big(\sum_{k = 1}^K W_k\big)^2}{\sum_{k = 1}^K W_k^2}.
\]
On the other hand, if we use the sample variance, $\si^2 = K(K - 1)^{-1}\wt
\si^2$, then we have $K_e \approx K_e''$, where
\[
  K_e'' = \frac K {1 + \si^2/\ol W^2}
    = \frac{K(K - 1)}{K - 1 + K \wt \si^2/\ol W^2}
    = \frac{K(K - 1)}{K - 1 + K(K/K_e' - 1)}
    = \left(\frac{K - 1}{K - K_e'/K}\right) K_e'.
\]

\subsection{Sequential imputation and the simulation measure}\label{S:seqImpSim}

Now fix $M \in \bN$. Let $Z = (Z_1, \ldots, Z_M)$ be an $S^M$-valued random
variable and let $z = (z_1, \ldots, z_M)$ denote an element of $S^M$. We adopt
the notation $\bmZ_m = (Z_1, \ldots, Z_m)$ and we use $\bmz_m = (z_1, \ldots,
z_m)$ for an element of $S^m$. Note that $\bmZ_M = Z$ and $\bmz_M = z$. We also
let $Y = (Y_1, \ldots, Y_M)$ be a $T^M$-valued random variable and adopt similar
notation in that case.

We think of $Y$ as observed values and $Z$ as unobserved. In this sense, $Z$ is
regarded as ``missing data.'' We wish to determine $\cL(Z \mid Y)$. Suppose,
however, that we are only able to determine $\cL(Z_m \mid \bmY_m, \bmZ_{m - 1})$
for $1 \le m \le M$. (By convention, a variable with a $0$ subscript is omitted.
Hence, when $m = 1$, we have $\cL(Z_m \mid \bmY_m, \bmZ_{m - 1}) = \cL (Z_1 \mid
Y_1)$.) We describe here a method of using $\cL(Z_m \mid \bmY_m, \bmZ_{m - 1})$
to approximate $\cL(Z \mid Y)$. This method is called sequential imputation and
first appeared in \cite{Kong1994}.

Consider, for the moment, the case $M = 2$. By conditioning on $Z_1$, we could
determine $\cL(Z \mid Y)$ sequentially, if we could compute
\begin{enumerate}[(i)]
  \item $\cL(Z_1 \mid Y)$ and
  \item $\cL(Z_2 \mid Y, Z_1)$.
\end{enumerate}
The second of these is available to us, but the first is not. Instead of (i), we
can only compute $\cL(Z_1 \mid Y_1)$. The idea in sequential imputation is to
use $\cL(Z_1 \mid Y_1)$ to simulate $Z_1$, then use this simulated value in (ii)
to determine the law of $Z_2$. We are substituting the missing data $Z_1$ with
its (incorrectly) simulated value. This kind of substitution is called
imputation. Since we are using $\cL(Z_1 \mid Y_1)$ instead of the correct
distribution in (i), we must combine this with the method of importance sampling
presented in Theorem \ref{T:imp-samp}.

To apply Theorem \ref{T:imp-samp}, we first construct the simulation measure
$\fm^*$. Let $\ga_m$ be a probability kernel from $T^m \times S^{m - 1}$ to $S$
with $Z_m \mid \bmY_m, \bmZ_{m - 1} \sim \ga_m(\bmY_m, \bmZ_{m - 1})$. Let
$\ga_m^*$ be the probability kernel from $T^M \times S^{m - 1}$ to $S$ given by
$\ga_m^*(y, \bmz_{m - 1}) = \ga_m(\bmy_m, \bmz_{m - 1})$. Note that $\ga_{M -
1}^*$ is a probability kernel from $T^M \times S^{M - 2}$ to $S$ and $\ga_M^*$
is a probability kernel from $T^M \times S^{M - 1}$ to $S$. Hence, $\ga_{M -
1}^* \ga_M^*$ is a probability kernel from $T^M \times S^{M - 2}$ to $S^2$.
Iterating this, if we define $\fm^* = \ga_1^* \cdots \ga_M^*$, then $\fm^*$ is a
probability kernel from $T^M$ to $S^M$.

In Theorem \ref{T:imp-samp}, we have $Z^* \mid Y \sim \fm^*(Y)$. Hence,
\begin{equation}\label{sim-iter}
  Z_m^* \mid Y, Z_1^*, \ldots, Z_{m - 1}^*
    \sim \ga_m^*(Y, Z_1^*, \ldots, Z_{m - 1}^*)
    = \ga_m(\bmY_m, Z_1^*, \ldots, Z_{m - 1}^*).
\end{equation}
In other words, the simulated vector $Z^* = (Z_1^*, \ldots, Z_M^*
)$ can be constructed sequentially using $\cL(Z_m \mid \bmY_m, \bmZ_{m - 1})$,
where in each step the missing data $\bmZ_{m - 1}$ is imputed with the
previously simulated values $Z_1^*, \ldots, Z_{m - 1}^*$. To prove
that Theorem \ref{T:imp-samp} applies in this situation, we must find a weight
function $w$ satisfying \eqref{impSampCon}.

\subsection{A simulation density}\label{S:sim-dens}

As noted earlier, sequential imputation first appeared in \cite{Kong1994}.
There, a weight function was constructed using density functions. The proof and
construction in \cite{Kong1994} did not specify the codomain of the random
variables $Y$ and $Z$, nor did it specify the measures with respect to which
they have joint and conditional densities. In Theorem \ref{T:Kong-Liu} below, we
give a rigorous formulation of the proof in \cite{Kong1994}. First, we clarify
the assumptions about the existence of densities, and then show how this relates
to the simulation measure $\fm^*$.

\begin{assum}\label{A:dens-exist}
  There exist $\si$-finite measures $\fn$ and $\wt \fn$ on $S$ and $T$,
  respectively, such that $\cL(Y, Z) \ll \wt \fn^M \times \fn^M$.
\end{assum}

If Assumption \ref{A:dens-exist} holds, then we may let $f = d\cL(Y, Z)/d(\wt
\fn^M \times \fn^M)$ be a density of $(Y, Z)$ with respect to $\wt \fn^M \times
\fn^M$. If we write $f$ with omitted arguments, it is assumed that they have
been integrated out. For example,
\[
  f(y_1, \bmz_m) = \int_{S^{M - m}} \int_{T^{M - 1}} {
    f(y, z)
  } \, \wt \fn^{M - 1}(dy_2 \cdots dy_M) \, \fn^{M - m}(dz_{m + 1} \cdots dz_m).
\]
In other words, such functions are the marginal densities. By changing $f$ on a
set of measure zero, we may assume $f \in [0, \infty)$ everywhere and if the
value of a marginal density at a point is $0$, then $f$ at that point is $0$ for
all values of the omitted arguments.

We use $|$ to denote conditional densities. For example,
\[
  f(z_{m + 1} \mid \bmy_m, \bmz_m) = \frac{
    f(\bmy_m, \bmz_{m + 1})
  }{
    f(\bmy_m, \bmz_m)
  }.
\]
As usual, we adopt the convention that a variable with a $0$ subscript is
omitted. For instance, if $m = 1$, then $f(\bmy_m, \bmz_{m - 1}) = f(y_1)$ and
$f(z_m \mid \bmy_m, \bmz_{m - 1}) = f(z_1 \mid y_1)$.

With this notation, we may write
\[
  \ga_m^*(y, \bmz_{m - 1}, dz_m) = {
    f(z_m \mid \bmy_m, \bmz_{m - 1}) \, \fn(dz_m)
  }.
\]
We also have
\begin{multline*}
  (\ga_{M - 1}^* \ga_M^*)(y, \bmz_{M - 2}, dz_{M - 1} \, dz_M) = {
    \ga_M^*(y, \bmz_{M - 1}, dz_M) \, \ga_{M - 1}^*(y, \bmz_{M - 2}, dz_{M - 1})
  }\\
  = {
    f(z_M \mid \bmy_M, \bmz_{M - 1})
    f(z_{M - 1} \mid \bmy_{M - 1}, \bmz_{M - 2})
    \, \fn(dz_M) \, \fn(dz_{M - 1})
  }.
\end{multline*}
Iterating this, we obtain $\fm^*(y, dz) = f^*(y, z) \, \fn^M(dz)$, where
\[
  f^*(y, z) = \prod_{m = 1}^M f(z_m \mid \bmy_m, \bmz_{m - 1}).
\]

\subsection{A proof using densities}\label{S:pfDens}

We now define the weight function and show that sequential imputation leads
asymptotically to $\cL(Z \mid Y)$. For $(y, z) \in T \times S$, define
\[
  w(y, z) = \prod_{m = 1}^M f(y_m \mid \bmy_{m - 1}, \bmz_{m - 1}).
\]
Define $Z^{*, k}$ and $\wt Z^K$ as in \eqref{impSampDef}.

\begin{thm}\label{T:Kong-Liu}
  If Assumption \ref{A:dens-exist} holds and $f(y) \in L^2(\wt \fn^M)$, then
  \begin{equation}\label{Kong-Liu}
    \cL(\wt Z^K \mid Z^{*, 1}, \ldots, Z^{*, K}, Y) \to \cL(Z \mid Y)
    \quad \text{a.s.}
  \end{equation}
\end{thm}

\begin{proof}
  By Theorem \ref{T:imp-samp}, it suffices to show that $\cL(Z \mid Y) \ll
  \fm^*(Y)$ a.s., $f(Y) > 0$ a.s., $E f(Y) < \infty$, and
  \[
    w(Y, \cdot) = f(Y) \, \frac {d\cL(Z \mid Y)} {d\fm^*(Y)} \quad \text{a.s.}
  \]
  Since $f(y)$ is the density of $Y$ with respect to $\wt \fn^M$, we have
  \[
    P(f(Y) = 0) = \int_{f^{-1}(\{0\})} f(y) \, \wt \fn^M(dy) = 0,
  \]
  so that $f(Y) > 0$ a.s. Since $f(y) \in L^2(\wt \fn^M)$, we also have
  \[
    E f(Y) = \int_{T^M} f(y)^2 \, \wt \fn^M(dy) < \infty.
  \]
  Finally,
  \[
    w(y, z) f^*(y, z) = \prod_{m = 1}^M {
      \frac{f(\bmy_m, \bmz_{m - 1})}{f(\bmy_{m - 1}, \bmz_{m - 1})} \,
      \frac{f(\bmy_m, \bmz_m)}{f(\bmy_m, \bmz_{m - 1})}
    } = f(y, z).
  \]
  Hence,
  \begin{multline*}
    \cL(Z \mid Y = y; dz) = f(z \mid y) \, \fn^M(dz)
    = \frac{f(y, z)}{f(y)} \, \fn^M(dz)\\
    = \frac{w(y, z)}{f(y)} \, f^*(y, z) \, \fn^M(dz)
    = \frac{w(y, z)}{f(y)} \, \fm^*(y, dz).
  \end{multline*}
  Therefore, $d\cL(Z \mid Y)/d\fm^*(Y) = w(Y, \cdot)/f(Y)$ a.s.
\end{proof}

\subsection{A proof without a simulation density}\label{S:pfNoDens}

We wish to apply sequential imputation to determine $\cL(\bmmu_M \mid \bmX_
{MN})$, using the computable distributions \eqref{1rowConDis}. In this case, we
would naturally take $Z_m = \mu_m$ and $Y_m = X_{mN} = (\xi_{m1}, \xi_{m2},
\ldots, \xi_{mN})$. In \cite{Liu1996}, Liu alleged to do exactly this in
the special case $S = \{0, 1\}$, using the results in \cite{Kong1994} as his
justification.

Unfortunately, sequential imputation---as it is presented in Theorem
\ref{T:Kong-Liu}---does not apply in this case. Namely, Assumption
\ref{A:dens-exist} is not satisfied. In fact, it is straightforward to verify
that, as long as $\vrho$ is not a point mass, the vector $Z = \bmmu_m$ has no
joint density with respect to any product measure.

Hence, the proof of Theorem \ref{T:Kong-Liu}, which is a rigorous presentation
of the proof in \cite{Kong1994}, does not justify the use of sequential
imputation in this setting. This includes not only the general setting that we
are working with, but also the special case $S = \{0, 1\}$ that was treated in
\cite{Liu1996}.

In this section, we give a new proof of \eqref{Kong-Liu}, in which we do not
require the joint densities of Assumption \ref{A:dens-exist}. In doing so, we
retroactively justify the results in \cite{Liu1996}, and also lay the
foundations for applying sequential imputation to the NDP on an arbitrary state
space $S$.

We continue to let the simulation measure $\fm^*$ be defined as in Section
\ref{S:seqImpSim}, but we must drop the assumption that $\fm^*$ has a density
with respect to a product measure. We cannot drop densities altogether, though,
since they are essential to defining the weight function.

\begin{assum}\label{A:densYexist}
  There exist $\si$-finite measures $\fn_1, \fn_2, \ldots, \fn_M$ and $\wt \fn$
  on $S, S^2, \ldots, S^M$ and $T$, respectively, such that $\cL(Y, \bmZ_m) \ll
  \wt \fn^M \times \fn_m$ for every $m$.
\end{assum}

Under Assumption \ref{A:densYexist}, we may let $f_m$ be a density of $(Y,
\bmZ_m)$ with respect to $\wt \fn^M \times \fn_m$. We adopt the same assumptions
and notational conventions for $f_m$ as we did for $f$ in Section
\ref{S:sim-dens}. For $(y, z) \in T^M \times S^M$, define
\begin{equation}\label{weightFunc}
  w(y, z) = f_1(y_1) \prod_{m = 1}^{M - 1} {
    f_m(y_{m + 1} \mid \bmy_m, \bmz_m)
  }.
\end{equation}
Define $Z^{*, k}$ and $\wt Z^K$ as in \eqref{impSampDef}.

\begin{thm}\label{T:seqImpGen}
  If Assumption \ref{A:densYexist} holds and $f_M(y) \in L^2(\wt \fn^M)$, then
  \[
    \cL(\wt Z^K \mid Z^{*, 1}, \ldots, Z^{*, K}, Y) \to \cL(Z \mid Y)
    \quad \text{a.s.}
  \]
\end{thm}

\begin{proof}
  The first part of the proof of Theorem \ref{T:Kong-Liu} carries over, so we
  need only show that
  \begin{equation}\label{KongLiu+}
    w(Y, \cdot) = f_M(Y) \, \frac {d\cL(Z \mid Y)} {d\fm^*(Y)} \quad \text{a.s.}
  \end{equation}
  Let $k \in \{1, \ldots, M - 1\}$ and let $A \in \cT$, $B \in \cS^k$, and $C
  \in \cS$.
  Then
  \begin{align*}
    P(Y_{k + 1} &\in A, \bmZ_k \in B, Z_{k + 1} \in C \mid \bmY_k, \bmZ_k)\\
    &= 1_B(\bmZ_k) P(Y_{k + 1} \in A, Z_{k + 1} \in C \mid \bmY_k, \bmZ_k)\\
    &= 1_B(\bmZ_k) E[{
      P(Y_{k + 1} \in A, Z_{k + 1} \in C \mid \bmY_{k + 1}, \bmZ_k)
    } \mid \bmY_k, \bmZ_k]\\
    &= 1_B(\bmZ_k) E[{
      1_A(Y_{k + 1}) \ga_{k + 1}(\bmY_{k + 1}, \bmZ_k, C)
    } \mid \bmY_k, \bmZ_k]\\
    &= 1_B(\bmZ_k) \int_A {
      \ga_{k + 1}(\bmY_k, y_{k + 1}, \bmZ_k, C)
      f_k(y_{k + 1} \mid \bmY_k, \bmZ_k)
    } \, \wt \fn(dy_{k + 1}).
  \end{align*}
  Hence,
  \begin{align*}
    P(Y_{k + 1} &\in A, \bmZ_k \in B, Z_{k + 1} \in C \mid \bmY_k)\\
    &= E\bigg[{
      1_B(\bmZ_k) \int_A {
        \ga_{k + 1}(\bmY_k, y_{k + 1}, \bmZ_k, C)
        f_k(y_{k + 1} \mid \bmY_k, \bmZ_k)
      } \, \wt \fn(dy_{k + 1})
    } \bigg| \bmY_k \bigg]\\
    &= \int_B \int_A {
      \ga_{k + 1}(\bmY_k, y_{k + 1}, \bmz_k, C)
      f_k(y_{k + 1} \mid \bmY_k, \bmz_k)
    } \, \wt \fn(dy_{k + 1}) f_k(\bmz_k \mid \bmY_k) \, \fn_k(d\bmz_k)\\
    &= \int_A \int_B \int_C {
      f_k(y_{k + 1} \mid \bmY_k, \bmz_k)
      \, \ga_{k + 1}(\bmY_k, y_{k + 1}, \bmz_k, dz_{k + 1})
      f_k(\bmz_k \mid \bmY_k)
      \, \fn_k (d\bmz_k)
      \, \wt \fn(dy_{k + 1})
    }.
  \end{align*}
  On the other hand,
  \[
    P(Y_{k + 1} \in A, \bmZ_k \in B, Z_{k + 1} \in C \mid \bmY_k)
      = \int_A \int_{B \times C} {
      f_{k + 1}(y_{k + 1}, \bmz_{k + 1} \mid \bmY_k)
    } \, \fn_{k + 1}(d\bmz_{k + 1}) \, \wt \fn(dy_{k + 1}).
  \]
  Hence,
  \begin{multline*}
    \int_{B \times C} {
      f_{k + 1}(y_{k + 1}, \bmz_{k + 1} \mid \bmY_k)
    } \, \fn_{k + 1}(d\bmz_{k + 1})\\
    = \int_B \int_C {
      f_k(y_{k + 1} \mid \bmY_k, \bmz_k)
      \, \ga_{k + 1}(\bmY_k, y_{k + 1}, \bmz_k, dz_{k + 1})
      f_k(\bmz_k \mid \bmY_k)
      \, \fn_k (d\bmz_k)
    }, 
  \end{multline*}
  for $\wt \fn$-a.e.\@ $y_{k + 1} \in T$. In particular, with probability one,
  we have
  \begin{align*}
    \cL(\bmZ_{k + 1} \mid \bmY_{k + 1}; d&\bmz_{k + 1})\\
    &= {
      f_{k + 1}(\bmz_{k + 1} \mid \bmY_{k + 1})
    } \, \fn_{k + 1}(d\bmz_{k + 1})\\
    &= \frac{
      f_{k + 1}(Y_{k + 1}, \bmz_{k + 1} \mid \bmY_k)
    }{
      f_{k + 1}(Y_{k + 1} \mid \bmY_k)
    } \, \fn_{k + 1}(d\bmz_{k + 1})\\
    &= \frac{f_{k + 1}(\bmY_k)}{f_{k + 1}(\bmY_{k + 1})} {
      f_k(Y_{k + 1} \mid \bmY_k, \bmz_k)
      \, \ga_{k + 1}(\bmY_{k + 1}, \bmz_k, dz_{k + 1})
      f_k(\bmz_k \mid \bmY_k)
      \, \fn_k (dz_k)
    }.
  \end{align*}
  Since $f_{k + 1}(\bmy_k)$ and $f_k(\bmy_k)$ are both densities of $\bmY_k$
  with respect to $\wt \fn^k$, we have $f_{k + 1}(\bmy_k) = f_k(\bmy_k)$, $\wt
  \fn^k$-a.e. In particular, $f_{k + 1}(\bmY_k) = f_k(\bmY_k)$ a.s. Thus,
  \begin{multline*}
    f_{k + 1}(\bmz_{k + 1} \mid \bmY_{k + 1}) \, \fn_{k + 1}(d\bmz_{k + 1})\\
    = \frac{f_k(\bmY_k)}{f_{k + 1}(\bmY_{k + 1})} {
      f_k(Y_{k + 1} \mid \bmY_k, \bmz_k)
      \, \ga_{k + 1}^*(Y, \bmz_k, dz_{k + 1})
      f_k(\bmz_k \mid \bmY_k)
      \, \fn_k (dz_k)
    },
  \end{multline*}
  almost surely. Note that $\ga_1^*(Y, dz_1) = \ga_1(Y_1, dz_1) = f_1(z_1 \mid
  Y_1)
  \, \fn_1(dz_1)$. Hence, starting with $k = M - 1$ and iterating backwards to
  $k = 1$,
  we obtain
  \[
    \cL(Z \mid Y; dz) = {
      \frac 1 {f_M(Z)} \, w(Y, Z) (\ga_1^* \cdots \ga_M^*)(Y, dz)
    }.
  \]
  Since $\fm^* = \ga_1^* \cdots \ga_M^*$, this proves \eqref{KongLiu+}.
\end{proof}

\section{Sequential imputation for the NDP}\label{S:seqImpNDP}

In this section, we apply sequential imputation, in the form of Theorem
\ref{T:seqImpGen}, to an array of samples from an NDP. In Theorem
\ref{T:seqImpGen}, we take $Z_m = \mu_m$ and we let $Y_m$ represent some
observations we have made about the samples $X_{mN}$.

A direct observation of the samples would be represented by taking $Y_m =
X_{mN}$. In general, though, we cannot treat the case $Y_m = X_{mN}$ with
sequential imputation. This is because Assumption \ref{A:densYexist} may fail.
Part of Assumption \ref{A:densYexist} is that $Y = (Y_1, \ldots, Y_M)$ has a
joint density with respect to some product measure. But when $Y_m = X_{mN}$,
this fails even in the case $M = 2$ and $N = 1$. More specifically, it is
straightforward to verify that if $\vrho$ is not discrete, then $(\xi_{11},
\xi_{21})$ has no joint density with respect to any product measure.

On the other hand, we can observe discrete functions of $X_{mN}$. This is
because Assumption \ref{A:densYexist} is trivially satisfied whenever $Y$ is
discrete. From an applied perspective, this is no restriction at all. Any
real-world measurement will have limits to its precision, meaning that only a
finite number of measurement outcomes are possible. We therefore assume, from
this point forward, that $Y_m$ is a discrete function of $X_{mN}$.

Section \ref{S:discrObs} contains our main result, Theorem \ref{T:seqImpNDP}.
This theorem shows how to use sequential imputation to compute $\cL(\bmmu_M \mid
Y)$. The chief challenge is to construct the simulated row distributions,
$\bmu_M^{*, k}$. According to \eqref{sim-iter}, these should be constructed
using the single-row conditional distributions, $\cL({\mu_m \mid \bmY_m,
\bmmu_{m - 1}})$. In Section \ref{S:cond1row}, we compute $\cL({\mu_m \mid
\bmY_m, \bmmu_{m - 1}})$. In Section \ref{S:genSims}, we use these to generate
$\bmu_M^{*, k}$. Finally, in Section \ref{S:pfSI-NDP}, we give the proof of
Theorem \ref{T:seqImpNDP}.

\subsection{Sequential imputation with discrete observations}\label{S:discrObs}

Let $T$ be a countable set, fix $N \in \bN$, and let $\ph_m: S^N \to T$. Define
$Y_m = \ph_m(X_{mN})$. We adopt the notation of Section \ref{S:seqImpSim}, so
that $Y = (Y_1, \ldots, Y_M)$, $\bmY_m = (Y_1, \ldots, Y_m)$, $y = (y_1, \ldots,
y_M) \in T^M$, and $\bmy_m = (y_1, \ldots, y_m) \in T^m$. We will apply Theorem
\ref{T:seqImpGen} with $M_1$ in place of $S$ and $\bmmu_M$ in place of $Z$. We
therefore change notation from $z$ to $\nu$. That is, $\nu = (\nu_1, \ldots,
\nu_M) \in M_1^M$ and $\bmnu_m = (\nu_1, \ldots, \nu_m) \in M_1^m$.

Let $\vrho_n = \cL(X_{mn})$, so that $\vrho_1 = \vrho$. Using \eqref{jntLawSamp}
with $\ep \vrho$ instead of $\ka \rho$ gives us a recursive way to compute
$\vrho_n$. In particular, for $B_n \in \cS^n$ and $B \in \cS$, we have
\begin{equation*}
  \vrho_{n + 1}(B_n \times B) = \frac \ep{\ep + n} \vrho_n(B_n) \vrho(B)
    + \frac 1{\ep + n} \sum_{i = 1}^n \vrho_n(B_n \cap \pi_i^{-1} B),
\end{equation*}
where $\pi_i: S^n \to S$ is the projection onto the $i$th co-ordinate.

For $m \in \{1, \ldots, M\}$ and $y_m \in T$, let $A_m = \ph_m^{-1}(\{y_m\}) \in
\cS^N$. Then $Y_m = y_m$ if and only if $X_{mN} \in A_m$. Therefore, the 
\emph{prior likelihoods}, $P(Y_m = y_m)$, satisfy
\begin{equation}\label{Y_m-dist}
  P(Y_m = y_m) = \vrho_N(A_m).
\end{equation}
Although the notation does not explicitly indicate it, we must remember that the
set $A_m$ depends on the vector $y_m$.

Now fix $y = (y_1, \ldots, y_M) \in T^M$. Using $y$, we will construct a
\emph{weighted simulation} of $\bmmu_M$, which is a pair $(t, \bmu_M^*)$, where
$t = \{t_{mi}: 1 \le m \le M, 1 \le i \le m\}$ is a triangular array of $[0,
\infty)$-valued random variables and $\bmu_M^* = (u_1^*, \ldots, u_M^*)$ is a
vector of $M_1$-valued random variables, all of which are independent of $Y$.
The rows of $t$, which we denote by $t_m = (t_{m1}, \ldots, t_{mm})$, are called
the \emph{row weights} of the weighted simulation, and the random measures
$u_m^*$ are called the \emph{simulated row distributions}. We construct $t_m$
and $u_m^*$ by recursion on $m$ as follows. Let
\begin{equation}\label{t-defn}
  t_{mi} =\begin{cases}
    (u_i^*)^N(A_m) &\text{if $1 \le i < m$},\\
    \ka \vrho_N(A_m) &\text{if $i = m$},
  \end{cases}
\end{equation}
and
\begin{equation}\label{u^*-defn}
  \cL(u_m^* \mid \bmu_{m - 1}^*) \propto t_{mm} \int_{S^N} {
    \cD\bigg(\ep \vrho + \sum_{n = 1}^N \de_{x_n}\bigg)
  } \, \vrho_N(dx \mid A_m) + \sum_{i = 1}^{m - 1} t_{mi} \de_{u_i^*},
\end{equation}
where $\bmu_{m - 1}^* = (u_1^*, \ldots, u_{m - 1}^*)$ and $\vrho_N(A \mid A_m) =
P(X_{mN} \in A \mid X_{mN} \in A_m)$. In other words,
\[
  P(u_m^* = u_i^* \mid \bmu_{m - 1}^*) = {
    \frac{t_{mi}}{t_{m1} + \cdots t_{mm}}
  },
\]
for $1 \le i < m$, and, with probability $t_{mm}/(t_{m1} + \cdots + t_{mm})$,
the random measure $u_m^*$ is independent of $\bmu_{m - 1}^*$ and has
distribution
\begin{equation}\label{u^*-def-pt}
  \int_{S^N} {
    \cD\bigg(\ep \vrho + \sum_{n = 1}^N \de_{x_n}\bigg)
  } \, \vrho_N(dx \mid A_m).
\end{equation}
The above is what the distribution of $\mu_m$ would be if we had only observed
the row $y_m$. (This is a consequence of \eqref{DirNoisy2}.)

Finally, having constructed the weighted simulation $(t, \bmu_M^*)$, we define
the \emph{total weight} of the weighted simulation to be
\begin{equation}\label{V-defn}
  V = \prod_{m = 1}^M \frac 1{\ka + m - 1} \sum_{i = 1}^m t_{mi}.
\end{equation}

\begin{thm}\label{T:seqImpNDP}
  Let $\{(t^k, \bmu_M^{*, k})\}_{k = 1}^K$ be $K$ independent weighted
  simulations as above, with corresponding total weights $V_k$. Then
  \begin{equation}\label{seqImpNDP1}
    \cL(\bmmu_M \mid Y = y) = \lim_{K \to \infty} \frac{
      \sum_{k = 1}^K V_k \de(\bmu_M^{*, k})
    }{
      \sum_{k = 1}^K V_k
    },
  \end{equation}
  where $\de(\bmu_M^{*, k})$ is the point mass measure on $M_1^M$ centered at
  $\bmu_M^{*, k}$. Consequently, if $\Phi$ is a measurable function on $M_1^M$
  taking values in a metric space and $P(\bmmu_M \in D^c \mid Y = y) = 1$, where
  $D \subseteq M_1^M$ is the set of discontinuities of $\Phi$, then
  \begin{equation}\label{seqImpNDP2}
    \cL(\Phi(\bmmu_M) \mid Y = y) = \lim_{K \to \infty} \frac{
      \sum_{k = 1}^K V_k \de(\Phi(\bmu_M^{*, k}))
    }{
      \sum_{k = 1}^K V_k
    }.
  \end{equation}
\end{thm}

The proof of Theorem \ref{T:seqImpNDP} will be given in Section \ref{S:pfSI-NDP}.

\begin{cor}\label{C:seqImpNDP}
  With the assumptions of Theorem \ref{T:seqImpNDP}, we have
  \begin{equation}\label{lawHighRow}
    \cL(\mu_{M + 1} \mid Y = y) = \lim_{K \to \infty} \frac 1{\ka + M} \left({
      \ka \cD(\ep \vrho) + \sum_{m = 1}^M \frac{
        \sum_{k = 1}^K V_k \de(u_m^{*, k})
      }{
        \sum_{k = 1}^K V_k
      }
    }\right).
  \end{equation}
  Consequently, if $\Phi$ is a measurable function on $M_1$ taking values in a
  metric space and $P(\mu_{M + 1} \in D^c \mid Y = y) = 1$, where $D \subseteq
  M_1$ is the set of discontinuities of $\Phi$, then
  \begin{equation}\label{lawHighRow2}
    \cL(\Phi(\mu_{M + 1}) \mid Y = y)
      = \lim_{K \to \infty} \frac 1{\ka + M} \left({
        \ka \cD(\ep \vrho) \circ \Phi^{-1} + \sum_{m = 1}^M \frac{
          \sum_{k = 1}^K V_k \de(\Phi(u_m^{*, k}))
        }{
          \sum_{k = 1}^K V_k
        }
      }\right).
  \end{equation}
\end{cor}

\begin{proof}
  Let $\Psi: M_1 \to \bR$ be continuous and bounded. Using \eqref{Dir-post} and
  the fact that $\mu_{M + 1}$ and $Y$ are conditionally independent given
  $\bmmu_M$, we have
  \begin{align*}
    E[\Psi(\mu_{M + 1}) \mid Y]
      &= E[E[\Psi(\mu_{M + 1}) \mid \bmmu_M, Y] \mid Y]\\
    &= E[E[\Psi(\mu_{M + 1}) \mid \bmmu_M] \mid Y]\\
    &= E\bigg[
      \frac{\ka}{\ka + M} \int_{M_1} \Psi(\nu) \, \cD(\ep \vrho, d\nu)
      + \frac 1{\ka + M} \sum_{m = 1}^M \Psi(\mu_m)
    \;\bigg|\;
      Y
    \bigg]\\
    &= \frac{\ka}{\ka + M} \int_{M_1} \Psi(\nu) \, \cD(\ep \vrho, d\nu)
      + \frac 1{\ka + M} \sum_{m = 1}^M E[\Psi(\mu_m) \mid Y]
  \end{align*}
  By \eqref{seqImpNDP2}, this gives
  \[
    E[\Psi(\mu_{M + 1}) \mid Y]
      = \frac{\ka}{\ka + M} \int_{M_1} \Psi(\nu) \, \cD(\ep \vrho, d\nu)
      + \frac 1{\ka + M} \sum_{m = 1}^M \lim_{K \to \infty} {
        \frac{
          \sum_{k = 1}^K V_k \Psi(u_m^{*, k})
        }{
          \sum_{k = 1}^K V_k
        }
      }.
  \]
  Since $\Psi$ was arbitrary, this proves \eqref{lawHighRow}, and
  \eqref{lawHighRow2} follows immediately.
\end{proof}

\subsection{Conditioning on a single row}\label{S:cond1row}

We will prove Theorem \ref{T:seqImpNDP} by applying Theorem \ref{T:seqImpGen}.
To do this, we must, among other things, compute the conditional distribution
$\ga_m$ described in Section \ref{S:seqImpSim}. This is done below, and the
result is presented in \eqref{compGamma}. In order to derive this result, we
begin by establishing some formulas that we will need later.

Note that $\si(\prod_{i = 1}^m \mu_i^n) \subseteq \si(\vpi, \bmmu_m) \subseteq
\si(\mu)$. Also, since $\{\xi_{ij}\}$ is row exchangeable, we have
$\cL(\bmX_{mn} \mid \mu) = \prod_{i = 1}^m \mu_i^n$. Therefore,
\begin{equation}\label{condOnBmu}
  \cL(\bmX_{mn} \mid \vpi, \bmmu_m) = \cL(\bmX_{mn} \mid \bmmu_m)
    = \prod_{i = 1}^m \mu_i^n.
\end{equation}
Now let $A \subseteq S^n$ and $B \subseteq M_1$ be Borel. By \eqref{condOnBmu},
\begin{align*}
  P(\mu_m \in B, X_{mn} \in A \mid \bmmu_{m - 1})
    &= E[1_B(\mu_m) P(X_{mn} \in A \mid \bmmu_m) \mid \bmmu_{m - 1}]\\
  &= E[1_B(\mu_m) \mu_m^n(A) \mid \bmmu_{m - 1}].
\end{align*}
By \eqref{Dir-post}, this gives
\begin{equation}\label{condLawMuX}
  P(\mu_m \in B, X_{mn} \in A \mid \bmmu_{m - 1})
    = \frac 1 {\ka + m - 1} \bigg(
      \ka E[1_B(\mu_m) \mu_m^n(A)] + \sum_{i = 1}^{m - 1} 1_B(\mu_i) \mu_i^n(A)
    \bigg).
\end{equation}
In particular, since $P(X_{mn} \in A) = E[\mu_m^n(A)]$, we have
\begin{equation}\label{condLawX}
  P(X_{mn} \in A \mid \bmmu_{m - 1})
    = \frac 1 {\ka + m - 1} \bigg(
      \ka P(X_{mn} \in A) + \sum_{i = 1}^{m - 1} \mu_i^n(A)
    \bigg).
\end{equation}

\begin{thm}\label{T:compGamma}
  Fix $m \in \{1, \ldots, M\}$. Let $\ga_m$ be the probability kernel from $T^m
  \times M_1^{m - 1}$ to $M_1$ with ${\mu_m \mid \bmY_m, \bmmu_{m - 1}} \sim
  \ga_m(\bmY_m, \bmmu_{m - 1})$. Fix $\bmy_m \in T^m$ and $\bmnu_{m - 1} \in M_1^
  {m - 1}$. For $1 \le i \le m$, let
  \begin{equation}\label{q^m-defn}
    q_i = q_i^m(\bmnu_{m - 1}) = \begin{cases}
      \nu_i^N(A_m) &\text{if $1 \le i < m$},\\
      \ka \vrho_N(A_m) &\text{if $i = m$},
    \end{cases}
  \end{equation}
  and let $p_i = q_i/(q_1 + \cdots + q_m)$. Then
  \begin{equation}\label{compGamma}
    \ga_m(\bmy_m, \bmnu_{m - 1}) = p_m \int_{S^N} {
      \cD\bigg(\ep \vrho + \sum_{j = 1}^N \de_{x_j}\bigg)
    } \, \vrho_N(dx \mid A_m) + \sum_{i = 1}^{m - 1} p_i \de_{\nu_i}.
  \end{equation}
\end{thm}

\begin{proof}
  Let $\ga$ be the probability kernel on the right-hand side of
  \eqref{compGamma} and let $B \subseteq M_1$ be Borel. We must show that $P
  (\mu_m \in B \mid \bmY_m, \bmmu_{m - 1}) = \ga(\bmY_m, \bmmu_{m - 1}, B)$.
  Since $\bmX_{m - 1, N}$ and $\mu_m$ are conditionally independent given
  $\bmmu_{m-1}$, it suffices to show that $P(\mu_m \in B \mid Y_m, \bmmu_{m -
  1}) = \ga(\bmY_m, \bmmu_{m - 1}, B)$.

  Define the kernel $\vth_m$ from $T \times M_1^{m - 1}$ to $M_1$ by
  \begin{equation}\label{th-defn}
    \vth_m(y_m, \bmnu_{m - 1}, d\nu_m) = \ka \nu_m^N(A_m) \, \cL(\mu_m; d\nu_m)
      + \sum_{i = 1}^{m - 1} \nu_m^N(A_m) \, \de_{\nu_i}(d\nu_m).
  \end{equation}
  Then \eqref{condLawMuX} gives
  \begin{equation}\label{lawWithTh}
    P(\mu_m \in B, Y_m = y_m \mid \bmmu_{m - 1}) = \frac{
      \vth(y_m, \bmmu_{m - 1}, B)
    }{\ka + m - 1}.
  \end{equation}
  Now let $C \in \si (Y_m, \bmmu_{m - 1})$. Without loss of generality, we may
  assume that $C$ is of the form $C = \{Y_m \in D\} \cap \{\bmmu_{m - 1} \in F\}$
  for some $D \subseteq T$ and some Borel $F \subseteq M_1^{m - 1}$. Then 
  \eqref{lawWithTh}
  gives
  \begin{align*}
    E[1_B(\mu_m) 1_C] &= P(\mu_m \in B, Y_m \in D, \bmmu_{m - 1} \in F)\\
    &= E\bigg[1_F(\bmmu_{m - 1}) \sum_{y_m \in D} {
      P(\mu_m \in B, Y_m = y_m \mid \bmmu_{m - 1})
    }\bigg]\\
    &= E\bigg[1_F(\bmmu_{m - 1}) \sum_{y_m \in D} {
      \frac{\vth(y_m, \bmmu_{m - 1}, B)}{\vth(y_m, \bmmu_{m - 1}, M_1)}
      P(Y_m = y_m \mid \bmmu_{m - 1})
    } \bigg],
  \end{align*}
  where in the last line we have used \eqref{lawWithTh} with $B = M_1$. Hence,
  \begin{align*}
    E[1_B(\mu_m) 1_C] &= \sum_{y_m \in D} E\bigg[1_F(\bmmu_{m - 1}) {
      \frac{\vth(y_m, \bmmu_{m - 1}, B)}{\vth(y_m, \bmmu_{m - 1}, M_1)}
      E[1_{\{y_m\}}(Y_m) \mid \bmmu_{m - 1}]
    } \bigg]\\
    &= \sum_{y_m \in D} E\bigg[ E\bigg[ 1_F(\bmmu_{m - 1}) {
      \frac{\vth(y_m, \bmmu_{m - 1}, B)}{\vth(y_m, \bmmu_{m - 1}, M_1)}
      1_{\{y_m\}}(Y_m) \;\bigg|\; \bmmu_{m - 1}
    } \bigg] \bigg]\\
    &= \sum_{y_m \in D} E\bigg[1_F(\bmmu_{m - 1}) {
      \frac{\vth(y_m, \bmmu_{m - 1}, B)}{\vth(y_m, \bmmu_{m - 1}, M_1)}
      1_{\{y_m\}}(Y_m)
    } \bigg].
  \end{align*}
  We can rewrite this is
  \begin{align*}
    E[1_B(\mu_m) 1_C] &= \sum_{y_m \in D} E\bigg[1_F(\bmmu_{m - 1}) {
      \frac{\vth(Y_m, \bmmu_{m - 1}, B)}{\vth(Y_m, \bmmu_{m - 1}, M_1)}
      1_{\{y_m\}}(Y_m)
    } \bigg]\\
    &= E\bigg[1_F(\bmmu_{m - 1}) {
      \frac{\vth(Y_m, \bmmu_{m - 1}, B)}{\vth(Y_m, \bmmu_{m - 1}, M_1)}
      \sum_{y_m \in D} 1_{\{y_m\}}(Y_m)
    } \bigg]\\
    &= E\bigg[1_F(\bmmu_{m - 1}) {
      \frac{\vth(Y_m, \bmmu_{m - 1}, B)}{\vth(Y_m, \bmmu_{m - 1}, M_1)} 1_D(Y_m)
    } \bigg]\\
    &= E\bigg[
      \frac{\vth(Y_m, \bmmu_{m - 1}, B)}{\vth(Y_m, \bmmu_{m - 1}, M_1)} 1_C
    \bigg].
  \end{align*}
  Hence,
  \[
    P(\mu_m \in B \mid Y_m, \bmmu_{m - 1})
      = \frac{\vth(Y_m, \bmmu_{m - 1}, B)}{\vth(Y_m, \bmmu_{m - 1}, M_1)}.
  \]
  It remains to show that $\ga(\bmy_m, \bmmu_{m - 1}) = \vth(y_m, \bmmu_{m - 1})
  / \vth(y_m, \bmmu_{m - 1}, M_1)$. Note that
  \begin{align*}
    P(\mu_m \in B, Y_m = y_m) &= E[1_B(\mu_m) P(X_{mN} \in A_m \mid \mu_m)]\\
    &= E[1_B(\mu_m) \mu_m^N(A_m)]\\
    &= \int_B \nu_m^N(A_m) \, \cL(\mu_m; d\nu_m),
  \end{align*}
  which shows that
  \[
    \cL(\mu_m \mid Y_m = y_m; d\nu_m)
      = \frac 1{P(Y_m = y_m)} \, \nu_m^N(A_m) \, \cL(\mu_m; d\nu_m).
  \]
  Thus, \eqref{th-defn} becomes
  \[
    \vth(y_m, \bmnu_{m - 1}) = \ka P(Y_m = y_m) \cL(\mu_m \mid Y_m = y_m)
      + \sum_{i = 1}^{m - 1} \nu_i^N(A_m) \de_{\nu_i}.
  \]
  By \eqref{DirNoisy2},
  \[
    \cL(\mu_m \mid Y_m = y_m) = \int_{S^N} {
      \cD\bigg(\ep \vrho + \sum_{n = 1}^N \de_{x_n}\bigg)
    } \, \cL(X_{mN} \mid Y_m = y_m; dx)
  \]
  Since $\{Y_m = y_m\} = \{X_{mN} \in A_m\}$ and $\vrho_N = \cL(X_{mN})$, we
  can combine these last two equations to arrive at
  \[
    \vth(y_m, \bmnu_{m - 1}) = \ka \vrho_N(A_m) \int_{S^N} {
      \cD\bigg(\ep \vrho + \sum_{n = 1}^N \de_{x_n}\bigg)
    } \, \vrho_N(dx \mid A_m) + \sum_{i = 1}^{m - 1} {
      \nu_i^N(A_m) \de_{\nu_i}
    }.
  \]
  It therefore follows from \eqref{compGamma} that $\ga(\bmy_m, \bmmu_{m -
  1}) = \vth(y_m, \bmmu_{m - 1}) / \vth(y_m, \bmmu_{m - 1}, M_1)$.
\end{proof}

\subsection{Generating the simulations}\label{S:genSims}

Having computed $\ga_m$ in Theorem \ref{T:compGamma}, we can now compute the
simulation measure $\fm^*$, described in Section \ref{S:seqImpSim}. In
\eqref{u^*-defn} and \eqref{V-defn}, the terms $u_m^*$ and $V$ depend on $y$. To
emphasize this dependence, we write $u_m^{*, k} = u_m^{*, k}(y)$ and $V_k = V_k
(y)$. Define $\mu_m^{*, k} = u_m^{*, k}(Y)$. Note that $u_m^{*, k}(y)$ is
independent of $Y$, whereas $\mu_m^{*, k}$ is not. The random measure $\mu_m^{*,
k}$ is playing the role of $Z_m^{*, k}$ in Section \ref{S:seqImpSim}. To show
that we have constructed $\mu_m^{*, k}$ correctly, we must show that $\{\bmmu_M^
{*, k}\}_{k = 1}^\infty \mid Y \sim \fm^*(Y)^\infty$. This is done below in
Proposition \ref{P:mu_k^*-iid}.

To prove Proposition \ref{P:mu_k^*-iid}, we use the following explicit
construction of $u_m^{*, k}(y)$. Define $H \subseteq \bR^m$ by $H = [0,
\infty)^m \setminus \{(0, \ldots, 0)\}$. Let
\[
  U = \{U_{mk}(t): 1 \le m \le M, 1 \le k \le K, t \in H\}
\]
be an independent collection of random variables, where $U_{mk}(t)$ takes values
in $\{1, \ldots, m\}$ and satisfies $P(U_{mk}(t) = i) = t_i/(t_1 + \cdots +
t_m)$. Let
\[
  \la = \{\la_{mk}: 1 \le m \le M, 1 \le k \le K\}
\]
be an independent collection of random measures on $S$, where $\la_{mk}$ is a
Dirichlet mixture satisfying
\begin{equation}\label{la_mk-defn}
  \la_{mk} \sim \int_{S^N} {
    \cD\bigg(\ep \vrho + \sum_{n = 1}^N \de_{x_n}\bigg)
  } \, \vrho_N(dx \mid A_m).
\end{equation}
Assume $U$, $\la$, and $Y$ are independent.

Define $t_m^k = (t_{m1}^k, \ldots, t_{mm}^k) \in \bR^m$ and $\th(m, k) \in \{1,
\ldots, m\}$ recursively as follows. Let $t_{11}^k = \ka \vrho_N(A_1)$ and $\th
(1, k) = 1$. For $m > 1$, let
\begin{equation}\label{t^mk-defn}
  t_{mi}^k =\begin{cases}
    \la_{\th(i, k), k}^N(A_m) &\text{if $1 \le i < m$},\\
    \ka \vrho_N(A_m) &\text{if $i = m$},
  \end{cases}
\end{equation}
and
\begin{equation}\label{theta-defn}
  \th(m, k) = \begin{cases}
    \th(U_{mk}(t^{mk}), k) &\text{if $1 \le U_{mk}(t^{mk}) < m$},\\
    m &\text{if $U_{mk}(t^{mk}) = m$}.
  \end{cases}
\end{equation}
With this construction, we may write $u_m^{*, k}(y) = \la_{\th(m, k), k}$.

In the proof of Proposition \ref{P:mu_k^*-iid}, we also use the notation $\cF
\vee \cG = \si(\cF \cup \cG)$, whenever $\cF$ and $\cG$ are $\si$-algebras on a
common set.

\begin{prop}\label{P:mu_k^*-iid}
  Let $\ga_m$ be as in Theorem \ref{T:compGamma} and $\fm^* = \ga_1^* \cdots
  \ga_M^*$, where $\ga_m^*(y, \bmnu_{m - 1}) = \ga_m(\bmy_m, \bmnu_{m - 1})$.
  Then $\{\bmmu_M^{*, k}\}_{k = 1}^\infty \mid Y \sim \fm^*(Y)^\infty$.
\end{prop}

\begin{proof}
  As noted in \eqref{sim-iter}, it suffices to show that $\mu_m^{*, k} \mid Y,
  \bmmu_{m - 1}^{*, k} \sim \ga_m(\bmY_m, \bmmu_{m - 1}^{*, k})$.

  We first note that if $\cF_{mk} = \si(U_{1k}, \ldots, U_{mk}, \la_{1k},
  \ldots, \la_{mk})$, then $t_m^k$ is $\cF_{m - 1, k}$-measurable and $\th(m,
  k)$ is $\cF_ {m - 1, k} \vee \si(U_{mk})$-measurable. This follows from
  \eqref{t^mk-defn} and \eqref{theta-defn} by induction. Also, by
  \eqref{theta-defn}, we have
  \begin{equation}\label{SI-NDP-pf2}
    u_m^{*, k}(y) = \begin{cases}
      u_i^{*, k}(y) &\text{if $U_{mk}(t^{mk}) = i < m$},\\
      \la_{mk}        &\text{if $U_{mk}(t^{mk}) =  m$}.
    \end{cases}
  \end{equation}
  Hence, $u_m^{*, k}(y)$ is $\cF_{mk}$-measurable. In particular, $U_{mk}$,
  $\la_ {mk}$, and $\bmu_{m - 1}^{*, k}(y)$ are independent.

  Now let $B \subseteq M_1$ be Borel and let $C \in \si (Y, \bmmu_{m - 1}^{*,
  k})$. Without loss of generality, we may assume that $C$ is of the form $C =
  \{Y \in D\} \cap \{\bmmu_{m - 1}^{*, k} \in F\}$ for some $D \subseteq T^M$
  and some Borel $F \subseteq M_1^{m - 1}$. We then have
  \begin{align}
    E[1_B(\mu_m^{*, k}) 1_C]
      &= P(Y \in D, \bmmu_m^{*, k} \in F \times B)\notag\\
    &= \sum_{y \in D} P(Y = y, \bmu_m^{*, k}(y) \in F \times B)\notag\\
    &= \sum_{y \in D} P(Y = y) P(\bmu_m^{*, k}(y) \in F \times B)\notag\\
    &= \sum_{y \in D} P(Y = y) E[1_F(\bmu_{m - 1}^{*, k}(y)) P(
      u_m^{*, k}(y) \in B \mid \bmu_{m - 1}^{*, k}(y)
    )] \label{SI-NDP-pf1}
  \end{align}
  Using \eqref{SI-NDP-pf2}, we obtain
  \begin{multline*}
    P(u_m^{*, k}(y) \in B \mid \bmu_{m - 1}^{*, k}(y)) = P(
      U_{mk}(t_m^k) = m, \la_{mk} \in B
    \mid
      \bmu_{m - 1}^{*, k}(y)
    )\\
    + \sum_{i = 1}^{m - 1} P(
      U_{mk}(t_m^k) = i, u_i^{*, k}(y) \in B
    \mid
      \bmu_{m - 1}^{*, k}(y)
    ).
  \end{multline*}
  From \eqref{t^mk-defn} and \eqref{q^m-defn}, it follows that $t_{mi}^k =
  q_i^m(\bmu_{m - 1}^{*, k}(y))$. Since $U_{mk}$, $\la_{mk}$, and $\bmu_{m -
  1}^{*, k}(y)$ are independent, the above becomes
  \[
    P(u_m^{*, k}(y) \in B \mid \bmu_{m - 1}^{*, k}(y))
      = p_m^m(\bmu_{m - 1}^{*, k}(y)) P(\la_{mk} \in B)
      + \sum_{i = 1}^{m - 1} {
        p_i^m(\bmu_{m - 1}^{*, k}(y)) \de_{u_i^{*, k}(y)}(B)
      }.
  \]
  It follows from \eqref{la_mk-defn} and \eqref{compGamma} that
  \[
    P(u_m^{*, k}(y) \in B \mid \bmu_{m - 1}^{*, k}(y))
      = \ga_m(\bmy_m, \bmu_{m - 1}^{*, k}(y), B).
  \]
  Substituting this into \eqref{SI-NDP-pf1} and noting that $\bmu_M^{*, k}(y)$
  and $Y$ are independent, we have
  \begin{align*}
    E[1_B(\mu_m^{*, k}) 1_C] &= \sum_{y \in D} P(Y = y) {
      E[1_F(\bmu_{m - 1}^{*, k}(y)) \ga_m(\bmy_m, \bmu_{m - 1}^{*, k}(y), B)]
    }\\
    &= \sum_{y \in D} E[1_{\{y\}}(Y) 1_F(\bmu_{m - 1}^{*, k}(y))\ga_m(
      \bmy_m, \bmu_{m - 1}^{*, k}(y), B
    )]\\
    &= \sum_{y \in D} E[1_{\{y\}}(Y) 1_F(\bmmu_{m - 1}^{*, k}) \ga_m(
      \bmY_m, \bmmu_{m - 1}^{*, k}, B
    )]\\
    &= E[1_D(Y) 1_F(\bmmu_{m - 1}^{*, k}) \ga_m(
      \bmY_m, \bmmu_{m - 1}^{*, k}, B
    )]\\
    &= E[\ga_m(\bmY_m, \bmmu_{m - 1}^{*, k}, B) 1_C],
  \end{align*}
  showing that $P(\mu_m^{*, k} \in B \mid Y, \bmmu_{m - 1}^{*, k}) =
  \ga_m(\bmY_m, \bmmu_{m - 1}^{*, k}, B)$.
\end{proof}

\subsection{Proof of the main result}\label{S:pfSI-NDP}

Having established Theorem \ref{T:compGamma} and Proposition 
\ref{P:mu_k^*-iid}, we are now ready to prove the main result.

\begin{proof}[Proof of Theorem \ref{T:seqImpNDP}]
  We apply Theorem \ref{T:seqImpGen}. Let $\ga_m$ and $\fm^*$ be as in
  Proposition \ref{P:mu_k^*-iid}.

  If $\wt \fn$ is counting measure on $T$ and $\fn_m = \cL(\bmmu_m)$, then
  $\cL(Y, \bmmu_m) \ll \wt \fn^M \times \fn_m$, so that Assumption
  \ref{A:densYexist} holds. Let $f_m$ be the density of $(Y, \bmmu_m)$ with
  respect to $\wt \fn^M \times \fn_m$, and recall the notational conventions of
  Section \ref{S:sim-dens}. Let $w(y, \nu)$ be given by \eqref{weightFunc}.

  Let $\bmmu_M^{*, k}$ be as in Proposition \ref{P:mu_k^*-iid}, so that
  $\{\bmmu_M^{*, k}\}_{k = 1}^\infty \mid Y \sim \fm^*(Y)^\infty$. We define the
  weights $W_k = w(Y, \bmmu_M^{*, k})$. We first prove that $W_k = V_k(Y)$,
  where, according to \eqref{V-defn}, we have
  \begin{equation}\label{V^k(y)-defn}
    V_k(y) = \prod_{m = 1}^M \frac 1{\ka + m - 1} \sum_{i = 1}^m t_{mi}^k.
  \end{equation}
  Note that
  \[
    f_1(y_1) = P(Y_1 = y_1) = \vrho_N(A_1).
  \]
  For the other factors in \eqref{weightFunc}, we use \eqref{condLawX} and
  \eqref{Y_m-dist}, and the fact that $\bmY_m$ and $Y_{m + 1}$ are conditionally
  independent given $\bmmu_m$ to obtain
  \begin{align*}
    P(Y_{m + 1} = y_{m + 1} \mid \bmY_m, \bmmu_m)
      &= P(X_{m + 1, N} \in A_{m + 1} \mid \bmmu_m)\\
    &= \frac 1 {\ka + m} \bigg(
      \ka \vrho_N(A_{m + 1}) + \sum_{i = 1}^m \mu_i^N(A_{m + 1})
    \bigg),
  \end{align*}
  so that
  \[
    f_m(y_{m + 1} \mid \bmy_m, \bmnu_m)
      = \frac \ka{\ka + m} \vrho_N(A_{m + 1})
      + \frac 1{\ka + m} \sum_{i = 1}^m \nu_i^N(A_{m + 1}).
  \]
  Substituting this into \eqref{weightFunc} gives
  \[
    w(y, \nu) = \vrho_N(A_1) \prod_{m = 1}^{M - 1} \bigg(
      \frac \ka{\ka + m} \vrho_N(A_{m + 1})
      + \frac 1{\ka + m} \sum_{i = 1}^m \nu_i^N(A_{m + 1})
    \bigg),
  \]
  which can be rewritten as
  \begin{equation}\label{NDP-weights}
    w(y, \nu) = \prod_{m = 1}^M \bigg(
      \frac \ka{\ka + m - 1} \vrho_N(A_m)
      + \frac 1{\ka + m - 1} \sum_{i = 1}^{m - 1} \nu_i^N(A_m)
    \bigg).
  \end{equation}
  In the proof of Proposition \ref{P:mu_k^*-iid}, we noted that $t_{mi}^k =
  q_i^m(\bmu_{m - 1}^{*, k}(y))$. Hence, by \eqref{V^k(y)-defn} and 
  \eqref{q^m-defn}, we have
  \begin{align*}
    V_k(y) &= \prod_{m = 1}^M {
      \frac 1{\ka + m - 1} \sum_{i = 1}^m q_i^m(\bmu_{m - 1}^{*, k}(y))
    }\\
    &= \prod_{m = 1}^M {
      \frac 1{\ka + m - 1} \bigg(
        \ka \vrho_N(A_m) + \sum_{i = 1}^{m - 1} (u_i^{*, k}(y))^N(A_m)
      \bigg)
    }.
  \end{align*}
  It follows from \eqref{NDP-weights} that $V_k(y) = w(y, \bmu_M^{*, k}(y))$, so
  that $W_k = V_k(Y)$.

  Finally, we construct $\wt \bmmu_M^K$ so that
  \begin{equation}\label{seqImpNDP3}
    \wt \bmmu_M^K \mid \bmmu_M^{*, 1}, \ldots, \bmmu_M^{*, K}, Y
      \propto \sum_{k = 1}^K W_k \de(\bmmu_M^{*, k}).
  \end{equation}
  Since $f_M(y) = P(Y = y)$, we have
  \[
    \int_{T^M} f_M(y)^2 \, \wt \fn^M(dy) = \sum_{y \in T^M} P(Y = y)^2
      \le \sum_{y \in T^M} P(Y = y) = 1,
  \]
  so that $f_M(y) \in L^2(\wt \fn^M)$. Hence, by Theorem \ref{T:seqImpGen},
  \[
    \cL(\bmmu_M \mid Y) = \lim_{K \to \infty} {
      \cL(\wt \bmmu_M^K \mid \bmmu_M^{*, 1}, \ldots, \bmmu_M^{*, K}, Y)
    }.
  \]
  Applying \eqref{seqImpNDP3} to the above gives
  \[
    \cL(\bmmu_M \mid Y) = \lim_{K \to \infty} \frac{
      \sum_{k = 1}^K W_k \de(\bmmu_M^{*, k})
    }{
      \sum_{k = 1}^K W_k
    }.
  \]
  Since $W_k = V_k(Y)$ and $\bmmu_M^{*, k} = \bmu_M^{*, k}(Y)$, this proves
  \eqref{seqImpNDP1}, and \eqref{seqImpNDP2} follows immediately.
\end{proof}

\section{Examples}\label{S:examples}

In this section, we present four hypothetical applications to illustrate the use
of the Theorem \ref{T:seqImpNDP} and Corollary \ref{C:seqImpNDP}. See
\url{https://github.com/jason-swanson/ndp} for the code used to generate the
simulations.

The framework for each of these examples was described in Section \ref{S:intro}.
In that framework, we interpret $\xi_{ij}$ as the $j$th action of the $i$th
agent. The space $S$ is therefore the set of possible actions.

All the examples in this section involve a finite state space $S$. But, as we
describe in Remark \ref{R:finObs}, this special case is easily generalized to
the case of an arbitrary $S$ in which our observations are made with limited
precision.

The outline of this section is as follows. In Section \ref{S:finNDP}, we
describe how Theorem \ref{T:seqImpNDP} and Corollary \ref{C:seqImpNDP} simplify
in the case that $S$ is finite. After that, the remainder of the section is
devoted to the examples. Our simplest example is in Section \ref{S:pennies}, and
concerns a malfunctioning pressed penny machine. Section \ref{S:tacks} presents
a similar example, but with significantly more data. This is the same example
treated in \cite{Liu1996} (originally considered in \cite{Beckett1994}) and is
concerned with the flicking of thumbtacks. Section \ref{S:reviews} applies the
NDP model to the analysis of Amazon reviews. The final example, found in Section
\ref{S:lboards}, is about video game leaderboards. To prepare for that example,
a custom prior distribution, which we call the ``gamer'' distribution, is
presented in Section \ref{S:gamer}.

\subsection{The case of a finite state space}\label{S:finNDP}

Let $L \ge 2$ be an integer and suppose that $S = \{0, \ldots, L - 1\}$. Let
$p_\l = \vrho(\{\l\})$, so that we may identify $\vrho$ with the vector $\bmp =
(p_0, \ldots, p_{L - 1})$. We assume that $p_\l > 0$ for all $\l \in S$.

Let $y = \{y_{mn}: 1 \le m \le M, 1 \le n \le N_m\}$ be a jagged array of
elements in $S$. The array $y$ denotes our observed data. That is, we observe
$\xi_{mn} = y_{mn}$ for $1 \le m \le M$ and $1 \le n \le N_m$, and we wish to
compute the conditional distribution of $\xi$ given these observations. Define
the row counts $\ol y = \{\ol y_{m\l}: 1 \le m \le M, 0 \le \l \le L - 1\}$ by
$\ol y_{m\l} = |\{n: y_{mn} = \l\}|$. Since $\xi$ is row exchangeable, all of
our calculations will depend on $y$ only through the array $\ol y$. We use $\ol
y_m$ to denote the vector $(\ol y_{m1}, \ldots, \ol y_{m, L - 1})$.

To apply Theorem \ref{T:seqImpNDP}, let $N = \max \{N_1, \ldots, N_M\}$ and $T =
\bigcup_{n = 1}^N S^n$. Let $\ph_m: S^N \to T$ be the projection onto the first
$N_m$ components, so that $\ph_m(x_1, \ldots, x_N) = (x_1, \ldots, x_ {N_m})$.
Then $Y = (Y_1, \ldots, Y_M)$, where $Y_m = \ph_m(X_{mN}) = X_{mN_m}$. Note that
$A_m = \{Y_m = y_m\} = \{X_{mN_m} = y_m\}$. Therefore, if we define $\th_{m\l} =
\mu_m(\{\l\})$, then the prior likelihoods satisfy
\[
  \vrho_N(A_m) = P(X_{mN_m} = y_m) = E[P(X_{mN_m} = y_m \mid \mu_m)]
    = E\bigg[\prod_{\l = 0}^{L - 1} \th_{m\l}^{\ol y_{m\l}}\bigg].
\]
From \eqref{DirProcFDD} it follows that $(\th_{m0}, \ldots, \th_{m,L - 1}) \sim
\Dir(\ep p_0, \ldots, \ep p_{L - 1})$. This gives
\begin{equation}\label{likelihoods}
  \vrho_N(A_m)
    = E\bigg[\prod_{\l = 0}^{L - 1} \th_{m\l}^{\ol y_{m\l}}\bigg]
    = \frac 1 {B(\ep \bmp)} \int_{\De^{L - 1}} {
      \prod_{\l = 0}^{L - 1} t_\l^{\ol y_{m\l} + \ep p_\l - 1}
    } \, dt
    = \frac {B(\ep \bmp + \ol y_m)} {B(\ep \bmp)},
\end{equation}
where $B(\bmx) = \Ga(\sum_{\l = 0}^L x_\l) ^{-1} \prod_{\l = 0}^L \Ga(x_\l)$ is
the multivariate Beta function.

Having computed the prior likelihoods, we turn our attention to the weighted
simulations. From \eqref{DirNoisy2}, it follows that
\eqref{u^*-def-pt} is equal to $\cL(\mu_m \mid Y_m = y_m)$. But $Y_m =
X_{mN_m}$, so by \eqref{Dir-Bayes} we can rewrite \eqref{t-defn} and
\eqref{u^*-defn} as
\[
  t_{mi} =\begin{cases}
    \prod_{n = 1}^{N_m} u_i^*(y_{mn}) &\text{if $1 \le i < m$},\\
    \ka \vrho_N(A_m) &\text{if $i = m$},
  \end{cases}
\]
and
\begin{equation}\label{simRowDist}
  \cL(u_m^* \mid \bmu_{m - 1}^*) \propto t_{mm} {
    \cD\bigg(\ep \vrho + \sum_{n = 1}^{N_m} \de_{y_{mn}}\bigg)
  } + \sum_{i = 1}^{m - 1} t_{mi} \de_{u_i^*}.
\end{equation}
If we define $\th_{m\l}^* = u_m^*(\{\l\})$, then we can rewrite the row weights
as
\[
  t_{mi} =\begin{cases}
    \prod_{\l = 0}^{L - 1} (\th_{i\l}^*)^{\ol y_{m\l}}
      &\text{if $1 \le i < m$},\\
    \ka \vrho_N(A_m) &\text{if $i = m$}.
  \end{cases}
\]
In this case, we can identify $u_m^*$ with the vector $\th_m^* = (\th_{m0}^*,
\ldots, \th_{m, L - 1}^*)$, and \eqref{simRowDist} becomes
\[
  \cL(\th_m^* \mid \bmth_{m - 1}^*) \propto t_{mm} \Dir(\ep \bmp + \ol y_m)
    + \sum_{i = 1}^{m - 1} t_{mi} \de_{\th_i^*},
\]
where $\bmth_m^* = (\th_1^*, \ldots, \th_m^*)$. In other words, for $i < m$, we
have $\th_m^* = \th_i^*$ with probabilty $t_{mi}$, and, with probability $t_
{mm}/(t_{m1} + \cdots + t_{mm})$, the random vector $\th_m^*$ is independent of
$\bmth_{m - 1}^*$ and has the Dirichlet distribution, $\Dir(\ep \bmp + \ol
y_m)$.

Finally, we define $V$, the total weight of the simulation. According to 
\eqref{V-defn}, the total weight should be
\begin{equation}\label{oldweight}
  \prod_{m = 1}^M \frac 1{\ka + m - 1} \sum_{i = 1}^m t_{mi}.
\end{equation}
But in Theorem \ref{T:seqImpNDP}, we see that the weights are all relative to
their sum, so we are free to multiply this value by any constant that does not
depend on $k$. Leaving it as it is will produce a very small number, on the
order of $1/M!$. For computational purposes, then, we multiply \eqref{oldweight}
by $c^M M!$, where $c$ is a nonrandom constant. The total weight of our
simulation is then
\begin{equation}\label{scaledWt}
  V = \prod_{m = 1}^M \frac {cm}{\ka + m - 1} \sum_{i = 1}^m t_{mi}
\end{equation}
We call $\log c$ the \emph{log scale factor} of the simulation. In the examples
covered later in this section, we used $c = 1$ unless otherwise specified.

Now, if $\th = (\th_{m\l}) \in \bR^{M \times L}$ and $\Phi: \bR^{M \times L} \to
\bR$ is continuous, then \eqref{seqImpNDP2} gives
\begin{equation}\label{agent}
  \cL(\Phi(\th) \mid Y = y) = \lim_{K \to \infty} \frac{
    \sum_{k = 1}^K V_k \de(\Phi(\bmth_M^{*, k}))
  }{
    \sum_{k = 1}^K V_k
  }.
\end{equation}
Similarly, if $\Phi: \bR^L \to \bR$ is continuous, then \eqref{lawHighRow2}
gives
\begin{equation}\label{newAgent}
  \cL(\Phi(\th_{M + 1}) \mid Y = y)
    = \lim_{K \to \infty} \frac 1{\ka + M} \left({
      \ka \Dir(\ep \bmp) \circ \Phi^{-1}
      + \sum_{m = 1}^M \frac{
        \sum_{k = 1}^K V_k \de(\Phi(\th_m^{*, k}))
      }{
        \sum_{k = 1}^K V_k
      }
    }\right).
\end{equation}

\begin{rmk}
  In the case $S = \{0, 1\}$, the base measure $\vrho$ is entirely determined by
  the number $p = \vrho(\{1\})$, and we may define a single row count for each
  row, $\ol y_m = |\{n: y_{mn} = 1\}|$. In this case, letting $a = \ep p$ and $b
  = \ep(1 - p)$, we can rewrite \eqref{likelihoods} as
  \[
    \vrho_N(A_m) = \frac{B(a + \ol y_m, b + N_m - \ol y_m)}{B(a, b)}.
  \]
  Defining $\th_m^* = u_m^*(\{1\})$, the row weights of the weighted simulations
  become
  \[
    t_{mi} =\begin{cases}
      (\th_i^*)^{\ol y_m}(1 - \th_i^*)^{N_m - \ol y_m}
        &\text{if $1 \le i < m$},\\
      \ka \vrho_N(A_m) &\text{if $i = m$},
    \end{cases}
  \]
  and \eqref{simRowDist} becomes
  \[
    \cL(\th_m^* \mid \bmth_{m - 1}^*)
      \propto t_{mm} \Bet(a + \ol y_m, b + N_m - \ol y_m)
      + \sum_{i = 1}^{m - 1} t_{mi} \de_{\th_i^*}.
  \]
\end{rmk}

\begin{rmk}\label{R:finObs}
  Let us return for the moment to the general setting, where $S$ is an arbitrary
  complete and separable metric space. Let $S'$ be another complete and
  separable metric space and let $\psi: S \to S'$ be measurable. Let $\xi' =
  \{\xi'_{ij}\}$, where $\xi'_{ij} = \psi(\xi_{ij})$. It is straightforward to
  verify that $\xi'$ is a row exchangeable array of $S'$-valued random variables
  whose row distribution generator $\vpi'$ satisfies $\vpi' \sim \cD(\ka \cD(\ep
  \vrho'))$, where $\vrho' = \vrho \circ \psi^{-1}$.

  We can apply this to $S' = \{0, \ldots, L - 1\}$, where $L \ge 2$. For each
  $\l \in S' = \{0, \ldots, L - 1\}$, choose $B_\l \in \cS$ so that $\{B_\l: \l
  \in S'\}$ is a partition of $S$. Define $\psi: S \to S'$ by $\psi = \sum_{\l =
  0}^{L - 1} \l 1_{B_\l}$ and let $\xi' = \{\xi'_{ij}\}$ where $\xi'_{ij} =
  \psi(\xi_{ij})$. Suppose we can only observe the process $\xi'$. That is, we
  can only observe the values of $\xi$ with enough precision to tell which piece
  of the partition those values lie in. Based on some set of these observations,
  we wish to make probabilistic inferences about $\xi'$. Since $\xi'$ is an
  array of samples from an NDP on $S'$, we may do this using the simplified
  formulas in this section.
\end{rmk}

\subsection{The pressed penny machine}\label{S:pennies}

Imagine a pressed penny machine, like those found in museums or tourist
attractions. For a fee, the machine presses a penny into a commemorative
souvenir. Now imagine the machine is broken, so that it mangles all the pennies
we feed it. Each pressed penny it creates is mangled in its own way. Each has
its own probability of landing on heads when flipped. In this situation, the
agents are the pennies and the actions are the heads and tails that they
produce.

Now suppose we create seven mangled pennies and flip each one 5 times, giving us
the results in Table \ref{Tb:coins}. Of the 35 flips, 23 of them (or about
65.7\%) were heads. In fact, 6 of the 7 coins landed mostly on heads. The
machine clearly seems predisposed to creating pennies that are biased towards
heads.

\begin{table}
  \centering
  \begin{tabular}{|c|c|c|c|c|c|}
    \hline
    \textbf{Coin \#}  & \textbf{1st Flip} & \textbf{2nd Flip} &
    \textbf{3rd Flip} & \textbf{4th Flip} & \textbf{5th Flip} \\ \hline
    1 & H & H & H & H & T \\ \hline
    2 & H & T & H & H & H \\ \hline
    3 & T & H & H & T & H \\ \hline
    4 & H & H & T & H & H \\ \hline
    5 & T & T & T & H & T \\ \hline
    6 & T & H & H & H & H \\ \hline
    7 & H & T & T & H & H \\ \hline
  \end{tabular}
  \caption{Results of flipping seven different mangled pennies}
  \label{Tb:coins}
\end{table}

Coin 5, though, produced only one head. Is this coin different from the others
and actually biased toward tails? Or was it mere chance that its flips turned
out that way? For instance, suppose all 7 coins had a 60\% chance of landing on
heads. In that case, there would still be a 43\% chance that at least one of
them would produce four tails. How should we balance these competing
explanations and arrive at some concrete probabilities?

One way to answer this is to model the example with an NDP as in Section
\ref{S:finNDP}. We take $L = 2$, so that $S = \{0, 1\}$, where $0$ represents
tails and $1$ represents heads. We then take $\ka = \ep = 1$ and $p_0 = p_1 =
1/2$. From the table above, we have $M = 7$, $N_m = 5$ for all $m$, and
\[
  y = \begin{pmatrix}
    1 & 1 & 1 & 1 & 0 \\
    1 & 0 & 1 & 1 & 1 \\
    0 & 1 & 1 & 0 & 1 \\
    1 & 1 & 0 & 1 & 1 \\
    0 & 0 & 0 & 1 & 0 \\
    0 & 1 & 1 & 1 & 1 \\
    1 & 0 & 0 & 1 & 1 
  \end{pmatrix}.
\]
With $K = 10000$, we generated the weighted simulations $(t^k, \bmth_7^{*, k})$
for $1 \le k \le K$, and computed their corresponding total weights, $V_k$. In
this case, the effective sample size of our simulations (denoted by $K_\ep''$ in
Section \ref{S:ess}) was approximately $6067$.

Before addressing Coin 5 directly, let us ask a different question. If we were
to get a new coin from this machine, how would we expect it to behave? The new
coin would have some random probability of heads, which is denoted by $\th_
{8, 1}$. Taking $\Phi: \bR^2 \to \bR$ to be the projection, $\Phi(x_0, x_1) =
x_1$, we can use \eqref{newAgent} to approximation the distribution of
$\th_{8, 1}$, giving $\cL(\th_{8, 1} \mid Y = y) \approx \nu$, where
\[
  \nu = \frac18 \bigg(\Bet(1/2, 1/2)
    + \sum_{m = 1}^7 \frac{
      \sum_{k = 1}^{10000} V_k \de(\th_{m, 1}^{*, k})
    }{
      \sum_{k = 1}^{10000} V_k
    }\bigg).
\]
Using this, we have $P(\xi_{8, 1} = 1 \mid Y = y) = E[\th_{8, 1} \mid Y = y]
\approx 0.633$, so that, given our observations, the first flip of a new coin
has about a 63.3\% chance of landing heads. To visualize the distribution of
$\th_{81}$ rather than simply its mean, we can plot the distribution function of
$\nu$. See Figure \ref{F:coins}(a) for a graph of $x \mapsto \nu((0, x])$.

\begin{figure}
  \centering
  \begin{subfigure}[b]{0.45\linewidth}
    \includegraphics[width=\linewidth]{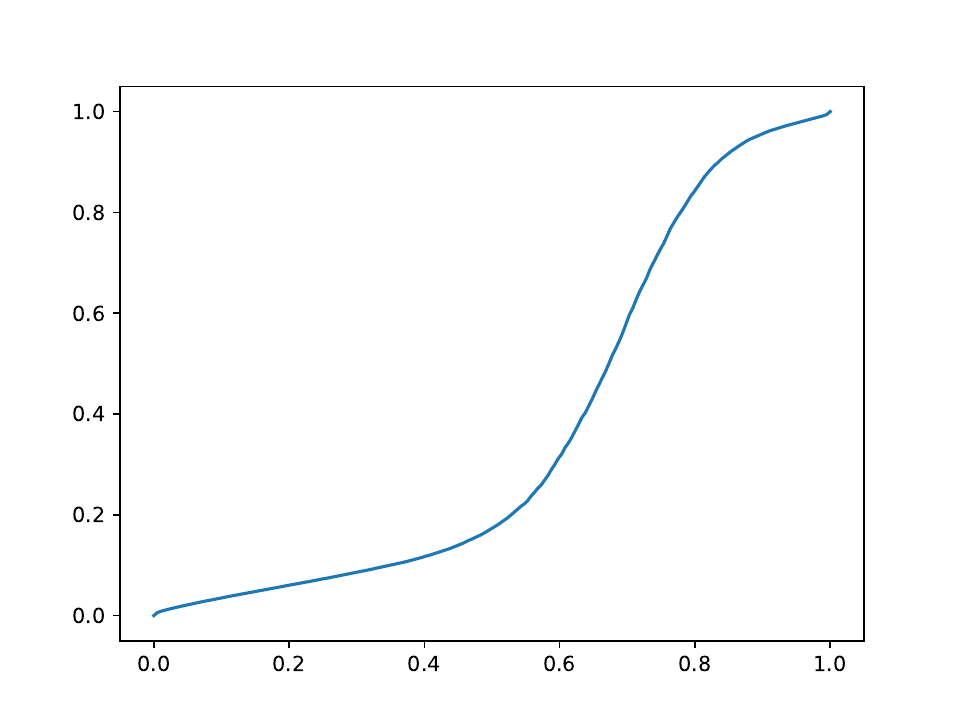}
    \caption{distribution function of $\th_{8, 1}$}
  \end{subfigure}
  \begin{subfigure}[b]{0.45\linewidth}
    \includegraphics[width=\linewidth]{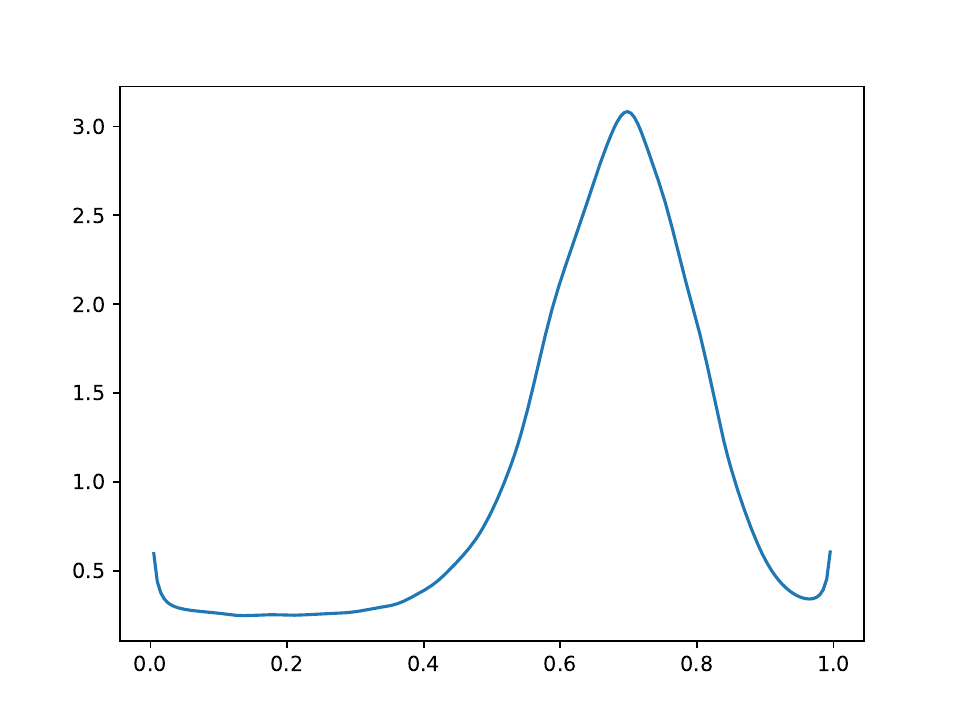}
    \caption{density of $\th_{8, 1}$ with $h \approx 0.119$}
  \end{subfigure}
  \begin{subfigure}[b]{0.45\linewidth}
    \includegraphics[width=\linewidth]{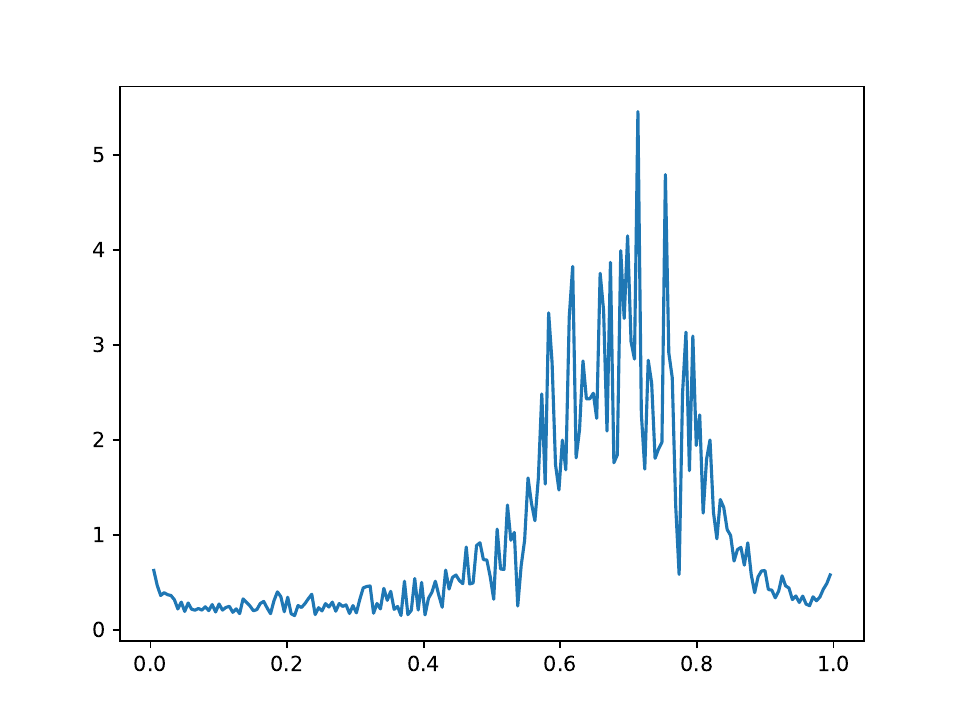}
    \caption{density of $\th_{8, 1}$ with $h = 0.001$}
  \end{subfigure}
  \begin{subfigure}[b]{0.45\linewidth}
    \includegraphics[width=\linewidth]{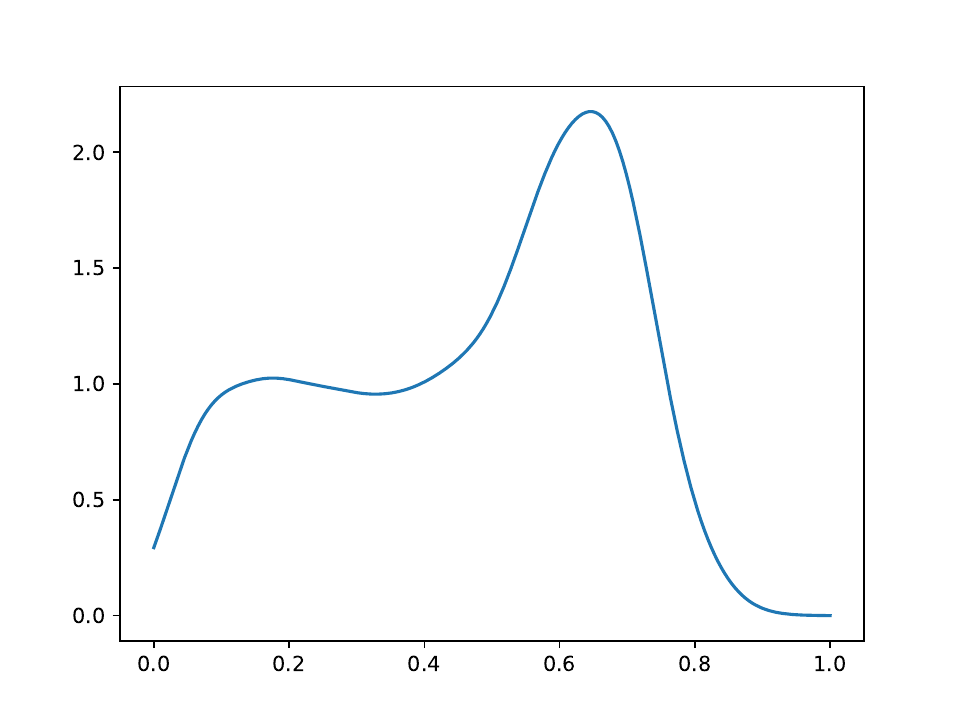}
    \caption{density of $\th_{5, 1}$}
  \end{subfigure}
  \caption{Approximate distribution and density functions for $\th_{m, \l}$}
  \label{F:coins}
\end{figure}

For a different visualization, we can plot an approximate density for $\nu$. The
measure $\nu$ has a discrete component, so we obtain an approximate density
using Gaussian kernel density estimation, replacing each point mass $\de_x$ by a
Gaussian measure with mean $x$ and standard deviation $h$, where $h$ is the
``bandwidth'' of the density estimation. For the measure $\nu$, we used Python's
\texttt{scipy.stats.gaussian\_kde} class to compute the bandwidth according to
Scott's Rule (see \cite{Scott1992}). In this case, we obtained $h \approx
0.119$, yielding the graph in Figure \ref{F:coins}(b). For a coarser estimate,
see Figure \ref{F:coins}(c), which uses $h = 0.001$. In all the remaining
examples in this section, we will default to using the bandwidth determined by
Scott's Rule.

Turning back to the question of Coin 5, if we define $\Phi: \bR^{7 \times 2}
\to \bR$ by $\Phi((x_{m, \l})) = x_{5, 1}$, then \eqref{agent} gives
\[
  \cL(\th_{5, 1} \mid Y = y) \approx \frac{
    \sum_{k = 1}^{10000} V_k \de(\th_{5, 1}^{*, k})
  }{
    \sum_{k = 1}^{10000} V_k
  }
\]
An approximate density for this measure is given in Figure \ref{F:coins}(d).
Using this, we can compute the probability that a sixth flip of Coin 5 lands on
heads, which is
\[
  P(\xi_{5, 6} = 1 \mid Y = y) = E[\th_{5, 1} \mid Y = y] \approx 0.461.
\]
We can also compute the probability that Coin 5 is biased toward tails, which is
given by $P(\th_{51} < 1/2 \mid Y = y) \approx 0.481$.

\subsection{Flicking thumbtacks}\label{S:tacks}

In \cite{Liu1996}, the following situation is considered. Imagine a box of 320
thumbtacks. We flick each thumbtack 9 times. If it lands point up, we call it a
success. Point down is a failure. Because of the imperfections, each thumbtack
has its own probability of success. The results (that is, the number of
successes) for these 320 thumbtacks are given by
\begin{align*}
  r = (
    & 7, 4, 6, 6, 6, 6, 8, 6, 5, 8, 6, 3, 3, 7, 8, 4,
      5, 5, 7, 8, 5, 7, 6, 5, 3, 2, 7, 7, 9, 6, 4, 6,\\
    & 4, 7, 3, 7, 6, 6, 6, 5, 6, 6, 5, 6, 5, 6, 7, 9,
      9, 5, 6, 4, 6, 4, 7, 6, 8, 7, 7, 2, 7, 7, 4, 6,\\
    & 2, 4, 7, 7, 2, 3, 4, 4, 4, 6, 8, 8, 5, 6, 6, 6,
      5, 3, 8, 6, 5, 8, 6, 6, 3, 5, 8, 5, 5, 5, 5, 6,\\
    & 3, 6, 8, 6, 6, 6, 8, 5, 6, 4, 6, 8, 7, 8, 9, 4,
      4, 4, 4, 6, 7, 1, 5, 6, 7, 2, 3, 4, 7, 5, 6, 5,\\
    & 2, 7, 8, 6, 5, 8, 4, 8, 3, 8, 6, 4, 7, 7, 4, 5,
      2, 3, 7, 7, 4, 5, 2, 3, 7, 4, 6, 8, 6, 4, 6, 2,\\
    & 4, 4, 7, 7, 6, 6, 6, 8, 7, 4, 4, 8, 9, 4, 4, 3,
      6, 7, 7, 5, 5, 8, 5, 5, 5, 6, 9, 1, 7, 3, 3, 5,\\
    & 7, 7, 6, 8, 8, 8, 8, 7, 5, 8, 7, 8, 5, 5, 8, 8,
      7, 4, 6, 5, 9, 8, 6, 8, 9, 9, 8, 8, 9, 5, 8, 6,\\
    & 3, 5, 9, 8, 8, 7, 6, 8, 5, 9, 7, 6, 5, 8, 5, 8,
      4, 8, 8, 7, 7, 5, 4, 2, 4, 5, 9, 8, 8, 5, 7, 7,\\
    & 2, 6, 2, 7, 6, 5, 4, 4, 6, 9, 3, 9, 4, 4, 1, 7,
      4, 4, 5, 9, 4, 7, 7, 8, 4, 6, 7, 8, 7, 4, 3, 5,\\
    & 7, 7, 4, 4, 6, 4, 4, 2, 9, 9, 8, 6, 8, 8, 4, 5,
      7, 5, 4, 6, 8, 7, 6, 6, 8, 6, 9, 6, 7, 6, 6, 6
  ).
\end{align*}
This data originally came from an experiment described in \cite{Beckett1994}. In
the original experiment, there were not 320 thumbtacks. Rather, there were 16
thumbtacks, 2 flickers, and 10 surfaces. We follow \cite{Liu1996}, however, in
treating the data as if it came from 320 distinct thumbtacks.

To model this example we take $L = 2$, so that $S = \{0, 1\}$, where $0$
represents failure (point down) and $1$ represents success (point up). To match
the modeling in \cite{Liu1996}, we take $\ep = 2$ and $p_0 = p_1 = 1/2$, so that
$\cD(\ep \vrho) \circ \pi_1^{-1} = \Bet(1, 1)$, where $\pi_1: M_1 \to [0, 1]$ is
the projection, $\nu \mapsto \nu(\{1\})$. We will use and compare two different
values of $\ka$ (which is denoted by $c$ in \cite{Liu1996}). For the data, we
have $M = 320$ and $N_m = 9$ for all $m$. Our row counts, $\ol y$, are given by
$\ol y_{m1} = r_m$ and $\ol y_{m0} = 9 - r_m$.

We first consider $\ka = 1$. As in \cite{Liu1996}, we generated $K = 10000$
weighted simulations. In this case, our effective sample size was approximately
244. (For comparison, in \cite{Liu1996}, Liu reported an effective sample size
of 227 for the case $\ka = 1$.) The unknown probability of success for a new
thumbtack is given by $\th_{321, 1}$, and \eqref{newAgent} gives
\[
  \cL(\th_{321, 1} \mid Y = y) \approx \frac1{321} \bigg(\Bet(1, 1)
    + \sum_{m = 1}^{320} \frac{
      \sum_{k = 1}^{10000} V_k \de(\th_{m, 1}^{*, k})
    }{
      \sum_{k = 1}^{10000} V_k
    }\bigg).
\]
An approximate density for this measure is given in Figure \ref{F:tacks1}(a).

We next consider $\ka = 10$, again using $K = 10000$, which generated an
effective sample size of about 388 (compared to 300 in \cite{Liu1996} for the
same value of $\ka$). This time, using \eqref{newAgent} gives
\[
  \cL(\th_{321, 1} \mid Y = y) \approx \frac1{330} \bigg(10\Bet(1, 1)
    + \sum_{m = 1}^{320} \frac{
      \sum_{k = 1}^{10000} V_k \de(\th_{m, 1}^{*, k})
    }{
      \sum_{k = 1}^{10000} V_k
    }\bigg).
\]
Note that in this second case, the simulated values $\th_{m, 1}^{*, k}$ and
their corresponding weights $V_k$ were all regenerated. An approximate density
for this measure is given in Figure \ref{F:tacks1}(b).

\begin{figure}
  \centering
  \begin{subfigure}[b]{0.45\linewidth}
    \includegraphics[width=\linewidth]{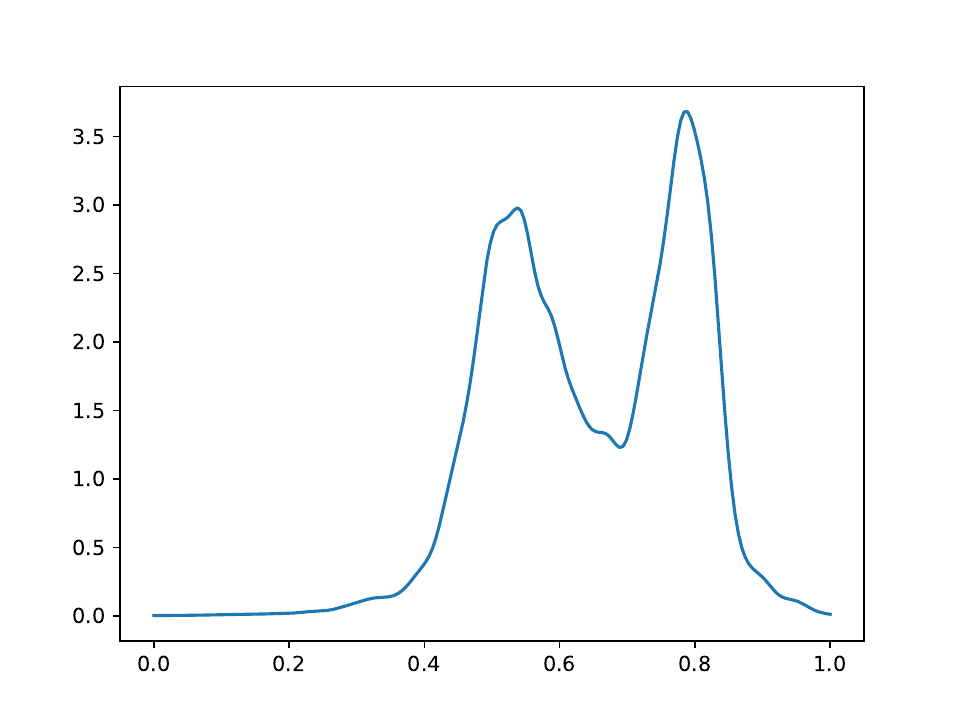}
    \caption{$\ka = 1$}
  \end{subfigure}
  \begin{subfigure}[b]{0.45\linewidth}
    \includegraphics[width=\linewidth]{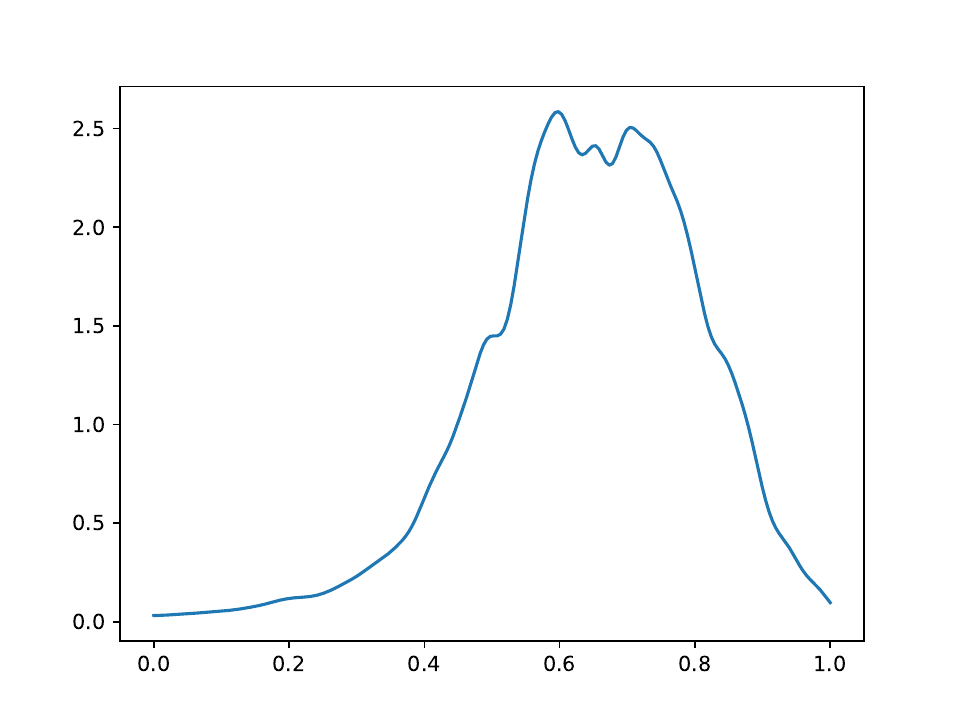}
    \caption{$\ka = 10$}
  \end{subfigure}
  \caption{Approximate density of $\cL(\th_{321, 1} \mid Y = y)$}
  \label{F:tacks1}
\end{figure}

As in the previous example, these approximate densities were constructed using
Gaussian kernel density estimation. Their respective bandwidths are $h \approx
0.105$ and $h \approx 0.096$. The graphs in Figure \ref{F:tacks1} are
qualitatively similar to their counterparts in \cite{Liu1996}, but with minor
differences. It is difficult, though, to make a direct comparison. Although
Gaussian kernel smoothing was also used in \cite{Liu1996}, details about the
smoothing were not provided. For instance, the bandwidths used to produce the
graphs in \cite{Liu1996} were not reported therein.

\subsection{Amazon reviews}\label{S:reviews}

The model in \cite{Liu1996} only covers agents with two possible actions, such
as coins and thumbtacks. The NDP, though, can handle agents whose range of
possible actions is arbitrary.

Imagine, then, that we discover a seller on Amazon that has 50 products. Their
products have an average rating of 2.4 stars out of 5. Some products have almost
100 ratings, while others have only a few. On average, the products have 23
ratings each. In this case, the agents are the products and the actions are the
ratings that each product earns. Each individual rating must be a whole number
of stars between 1 and 5, inclusive. Hence, each action has 5 possible outcomes.
The data used for this hypothetical seller is given in Table \ref{Tb:reviews}.

\begin{small}
  \begin{center}
    \begin{longtable}{|c|c|c|c|c|c|c|c|}
      \hline
      \textbf{product \#} & \textbf{1 star}   & \textbf{2 stars}  &
      \textbf{3 stars}    & \textbf{4 stars}  & \textbf{5 stars}  &
      \textbf{\# reviews} & \textbf{average}  \\ \hline
      \endhead
      1  & 9  & 25 & 15 & 41 & 0  & 90 & 2.98 \\ \hline
      2  & 21 & 28 & 18 & 1  & 3  & 71 & 2.11 \\ \hline
      3  & 16 & 11 & 21 & 11 & 0  & 59 & 2.46 \\ \hline
      4  & 3  & 9  & 37 & 0  & 3  & 52 & 2.83 \\ \hline
      5  & 11 & 0  & 36 & 0  & 5  & 52 & 2.77 \\ \hline
      6  & 16 & 16 & 4  & 15 & 0  & 51 & 2.35 \\ \hline
      7  & 30 & 3  & 15 & 0  & 0  & 48 & 1.69 \\ \hline
      8  & 12 & 9  & 17 & 1  & 7  & 46 & 2.61 \\ \hline
      9  & 13 & 13 & 18 & 1  & 0  & 45 & 2.16 \\ \hline
      10 & 23 & 2  & 0  & 14 & 0  & 39 & 2.13 \\ \hline
      11 & 11 & 4  & 6  & 7  & 10 & 38 & 3.03 \\ \hline
      12 & 6  & 3  & 21 & 0  & 5  & 35 & 2.86 \\ \hline
      13 & 14 & 9  & 0  & 5  & 2  & 30 & 2.07 \\ \hline
      14 & 4  & 25 & 0  & 0  & 0  & 29 & 1.86 \\ \hline
      15 & 8  & 7  & 2  & 10 & 0  & 27 & 2.52 \\ \hline
      16 & 5  & 4  & 6  & 10 & 0  & 25 & 2.84 \\ \hline
      17 & 6  & 10 & 9  & 0  & 0  & 25 & 2.12 \\ \hline
      18 & 11 & 1  & 2  & 3  & 7  & 24 & 2.75 \\ \hline
      19 & 20 & 3  & 0  & 0  & 0  & 23 & 1.13 \\ \hline
      20 & 6  & 9  & 4  & 2  & 1  & 22 & 2.23 \\ \hline
      21 & 5  & 1  & 3  & 8  & 1  & 18 & 2.94 \\ \hline
      22 & 9  & 1  & 5  & 2  & 1  & 18 & 2.17 \\ \hline
      23 & 5  & 7  & 3  & 1  & 1  & 17 & 2.18 \\ \hline
      24 & 0  & 3  & 12 & 0  & 2  & 17 & 3.06 \\ \hline
      25 & 1  & 11 & 1  & 3  & 1  & 17 & 2.53 \\ \hline
      26 & 0  & 3  & 0  & 6  & 7  & 16 & 4.06 \\ \hline
      27 & 2  & 2  & 8  & 3  & 1  & 16 & 2.94 \\ \hline
      28 & 6  & 5  & 1  & 3  & 0  & 15 & 2.07 \\ \hline
      29 & 6  & 6  & 1  & 2  & 0  & 15 & 1.93 \\ \hline
      30 & 0  & 8  & 2  & 4  & 0  & 14 & 2.71 \\ \hline
      31 & 8  & 5  & 1  & 0  & 0  & 14 & 1.5  \\ \hline
      32 & 5  & 0  & 8  & 0  & 1  & 14 & 2.43 \\ \hline
      33 & 0  & 0  & 13 & 0  & 0  & 13 & 3    \\ \hline
      34 & 5  & 4  & 1  & 2  & 0  & 12 & 2    \\ \hline
      35 & 6  & 2  & 0  & 3  & 0  & 11 & 2    \\ \hline
      36 & 4  & 7  & 0  & 0  & 0  & 11 & 1.64 \\ \hline
      37 & 0  & 1  & 6  & 4  & 0  & 11 & 3.27 \\ \hline
      38 & 5  & 5  & 0  & 1  & 0  & 11 & 1.73 \\ \hline
      39 & 5  & 6  & 0  & 0  & 0  & 11 & 1.55 \\ \hline
      40 & 1  & 2  & 2  & 4  & 1  & 10 & 3.2  \\ \hline
      41 & 4  & 1  & 3  & 1  & 0  & 9  & 2.11 \\ \hline
      42 & 4  & 1  & 1  & 0  & 0  & 6  & 1.5  \\ \hline
      43 & 3  & 1  & 0  & 1  & 0  & 5  & 1.8  \\ \hline
      44 & 3  & 0  & 1  & 0  & 0  & 4  & 1.5  \\ \hline
      45 & 1  & 2  & 0  & 0  & 0  & 3  & 1.67 \\ \hline
      46 & 0  & 1  & 2  & 0  & 0  & 3  & 2.67 \\ \hline
      47 & 2  & 1  & 0  & 0  & 0  & 3  & 1.33 \\ \hline
      48 & 0  & 0  & 2  & 0  & 0  & 2  & 3    \\ \hline
      49 & 0  & 1  & 0  & 1  & 0  & 2  & 3    \\ \hline
      50 & 0  & 0  & 1  & 1  & 0  & 2  & 3.5  \\ \hline
      \caption{Reviews for 50 different products from a given seller}
      \label{Tb:reviews}
    \end{longtable}
  \end{center}
\end{small}

To model this example we take $L = 5$, so that $S = \{0, 1, 2, 3, 4\}$, where
$\l \in S$ represents an $(\l + 1)$-star review. We take $\ka = 10$, $\ep = 5$,
and $p_\l = 1/5$ for each $\l \in S$. For the data, we have $M = 50$ and the
number $N_m$ is the total number of reviews given to the $m$th product. For our
row counts, the number $\ol y_{m\l}$ is the total number of $(\l + 1)$-star
reviews given to the $m$th product. In this example, we generated $K = 100000$
weighted simulations, and obtained an effective sample size of about 561. In
computing the simulation weights as in \eqref{scaledWt}, we used a log scale
factor of 28.8.

As with the pressed penny machine, we begin by considering a hypothetical new
product from this seller. The quality of this 51st product can be characterized
by the vector $\th_{51} = (\th_{51, 0}, \th_{51, 1}, \th_{51, 2}, \th_{51, 3},
\th_{51, 4})$, since $\th_{51, \l}$ is the (unknown) probability that the
product will receive an $(\l + 1)$-star review. The long-term average rating of
this product over many reviews will be $A(\th_{51})$, where $A(x) = \sum_{\l =
0}^4 \l x_\l$. According to \eqref{newAgent}, we have
\[
  \cL(A(\th_{51}) \mid Y = y) \approx \frac 1{60} \left({
      10 \Dir(1, 1, 1, 1, 1) \circ A^{-1}
      + \sum_{m = 1}^{50} \frac{
        \sum_{k = 1}^{100000} V_k \de(A(\th_m^{*, k}))
      }{
        \sum_{k = 1}^{100000} V_k
      }
    }\right).
\]
Using this, we have $E[A(\th_{51}) \mid Y = y] \approx 2.54$, meaning that the
expected long-term average rating of a new product is a little more than 2.5.
For a more informative look at the quality of a new product, an approximate
density for $\cL(A(\th_{51}) \mid Y = y)$ is given in Figure \ref{F:reviews}(a).

\begin{figure}
  \centering
  \begin{subfigure}[b]{0.45\linewidth}
    \includegraphics[width=\linewidth]{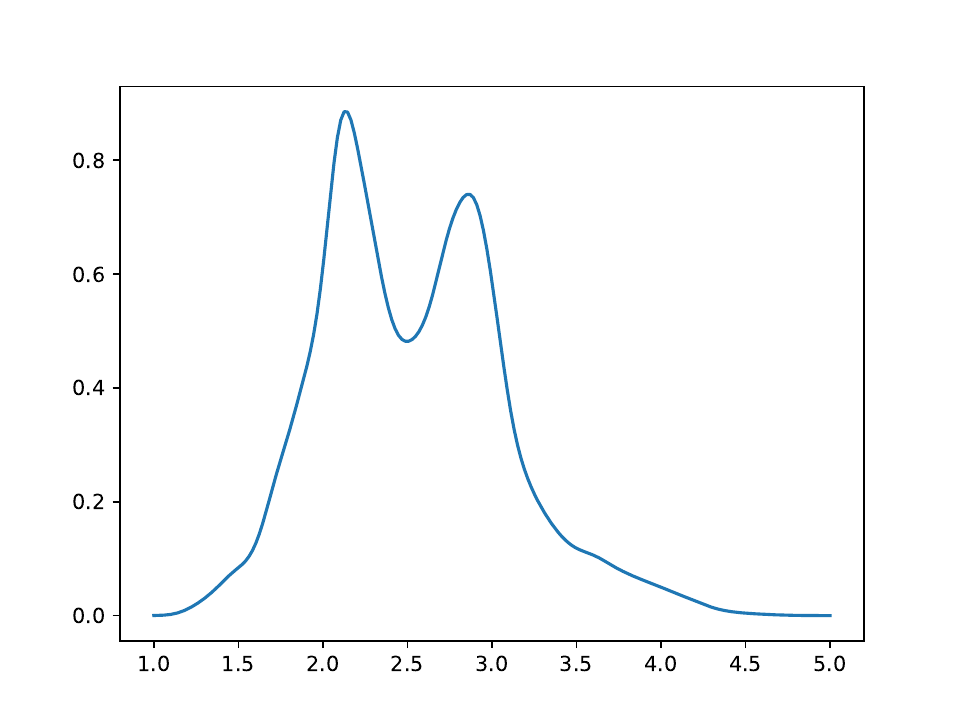}
    \caption{$m = 51$, mean: 2.54}
  \end{subfigure}
  \begin{subfigure}[b]{0.45\linewidth}
    \includegraphics[width=\linewidth]{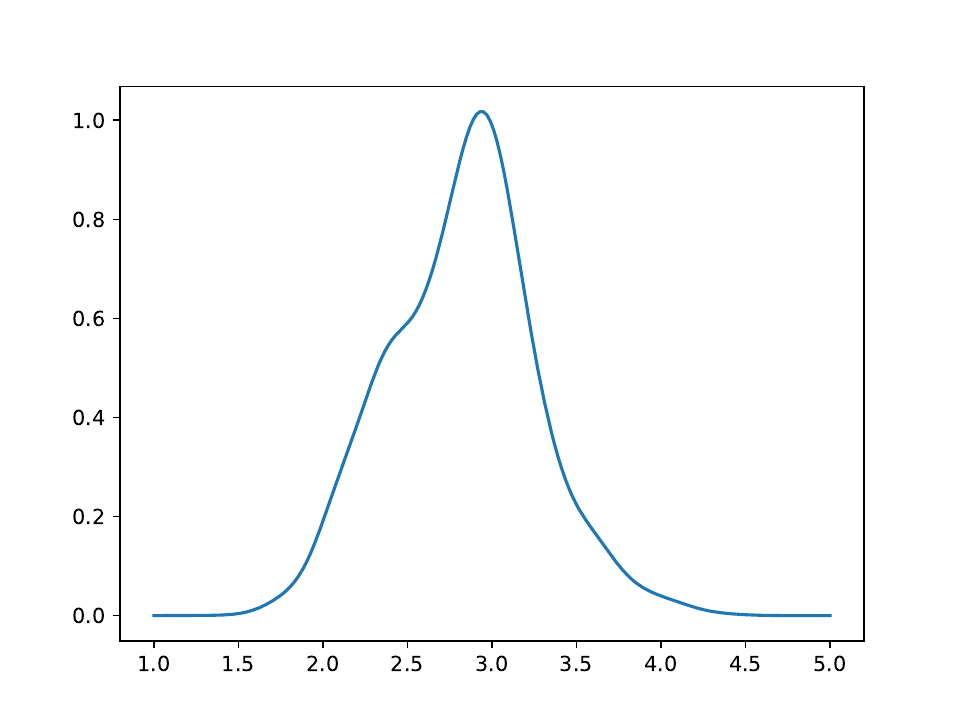}
    \caption{$m = 50$, mean: 2.83}
  \end{subfigure}
  \begin{subfigure}[b]{0.45\linewidth}
    \includegraphics[width=\linewidth]{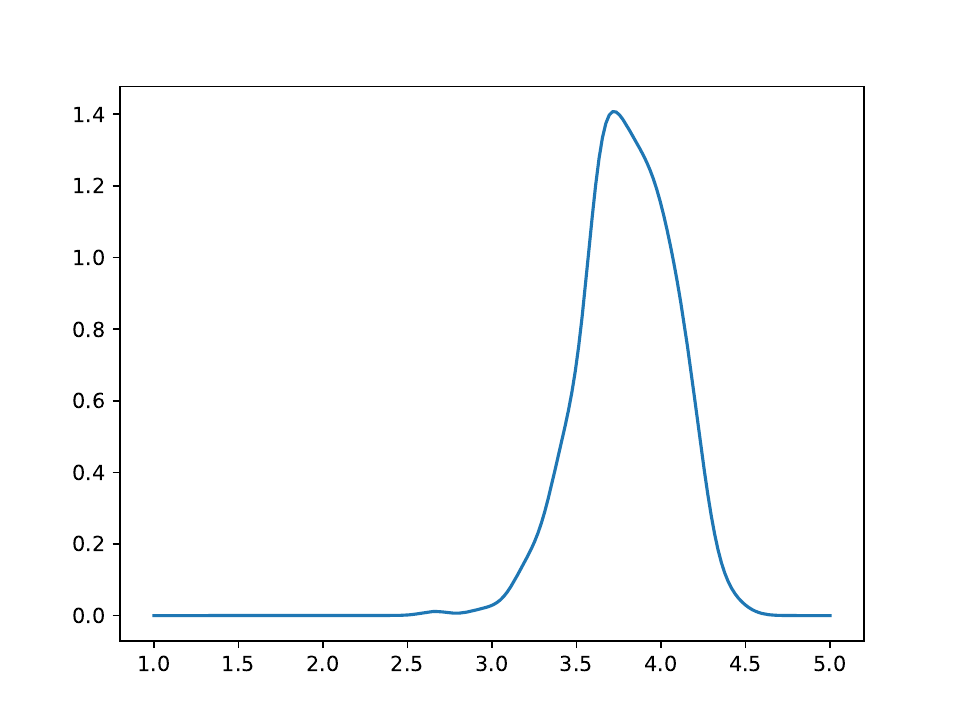}
    \caption{$m = 26$, mean: 3.8}
  \end{subfigure}
  \caption{Approximate densities for $\cL(A(\th_{m}) \mid Y = y)$}
  \label{F:reviews}
\end{figure}

This graph in Figure \ref{F:reviews}(a) shows a bimodal distribution, meaning
that we can expect the average ratings of future products to cluster around 2
and 3 stars.

After having considered a hypothetical new product, we turn our attention to the
50 products that have already received reviews. Consider, for instance, the 50th
product. This product has a 3.5-star average rating, but only 2 reviews. To see
the effect of these 2 reviews on the expected long-term rating, we apply 
\eqref{agent} with $\Phi((x_{m\l})) = A(x_{50})$ to obtain
\[
  \cL(A(\th_{50}) \mid Y = y) \approx \frac{
    \sum_{k = 1}^{100000} V_k \de(A(\th_{50}^{*, k}))
  }{
    \sum_{k = 1}^{100000} V_k
  }.
\]
This gives $E[A(\th_{50}) \mid Y = y] \approx 2.83$, and an approximate density
for $\cL(A(\th_{50}) \mid Y = y)$ is given in Figure \ref{F:reviews}(b).
According to the model, the 50th product's two reviews (a 3-star and a 4-star
review) have transformed the graph in Figure \ref{F:reviews}(a) to the graph in 
\ref{F:reviews}(b), and increased its expected long-term average rating from
2.54 to 2.83.

We can similarly look at the 26th product. This product has an average rating of
4.06, but it only has 16 reviews. Using \eqref{agent} as above, we obtain $E[A
(\th_{26}) \mid Y = y] \approx 3.8$, and an approximate density for $\cL(A(\th_
{26}) \mid Y = y)$ as given in Figure \ref{F:reviews}(c).

\subsection{The gamer distribution}\label{S:gamer}

In our previous examples, we took $\bmp$ to be a uniform measure. That is, we
took $p_\l = 1/L$ for all $\l$. Our final example will be presented in Section
\ref{S:lboards} and is concerned with video game leaderboards. In that example,
to have plausible results that match our intuition about video games, it will
not be sufficient to let $\bmp$ be uniform. Instead, we will construct $\bmp$
from the continuous distribution described in this section.

Let $r$, $c$, and $\al$ be positive real numbers. A nonnegative random variable
$X$ is said to have the \emph{gamer distribution} with parameters $r$, $c$, and
$\al$, denoted by $X \sim \Gamer(r, c, \al)$, if $X$ has density
\begin{equation}\label{gamrDens}
  f(x) = \frac {r c^r}{\al^r \, \Ga(\al)} \,
    x^{-r - 1} \, \int_0^{\al x/c} y^{\al + r - 1} e^{-y} \, dy,
\end{equation}
for $x > 0$. The fact that this is a probability density function is a consequence of 
Proposition \ref{P:makeGamr} below. Note that if $X \sim \Gamer(r, c, \al)$ and
$s > 0$, then $sX \sim \Gamer(r, sc, \al)$.

The gamer distribution is meant to model the score of a random player in a
particular single-player game. The game is assumed to have a structure in which
the player engages in a sequence of activities that can result in success or
failure. Successes increase the player's score. Failures bring the player
closer to a termination event, which causes the game to end.

The distribution of scores at the higher end of the player skill spectrum has a
power law decay with exponent $r$. More specifically, there is constant $K$ such
that $P(X > x) \approx K x^ {-r}$ for large values of $x$. For small values of
$x$, the distribution of $X$ looks like a gamma distribution.

The parameter $c$ indicates the average score of players at the lower end of the
skill spectrum, which make up the bulk of the player base. The parameter $\al$
is connected to the structure of the game. Higher values of $\al$ indicate a
more forgiving game in which the termination event is harder to trigger. See
below for more on the meaning of these parameters.

The gamer distribution can be seen as a mixture of gamma distributions, where
the mixing distribution is Pareto. More specifically, it is straightforward to
prove the following.

\begin{prop}\label{P:makeGamr}
  Let $r, c > 0$ and let $M$ have a Pareto distribution with minimum value $c$
  and tail index $r$. That is, $P(M > m) = (m/c)^{-r}$ for $m > c$. If $X \mid M
  \sim \Gam(\al, \al/M)$, then $X \sim \Gamer(r, c, \al)$.
\end{prop}

According to Proposition \ref{P:makeGamr}, the parameter $r$ is the tail index
of the mean player scores in the population. However, it is also the tail index
of the raw player scores. To see this, let $\ga(\be, u) = \int_0^u y^ {\be - 1}
e^{-y} \, dy$ denote the lower incomplete gamma function. Then \eqref{gamrDens}
can be rewritten as
\begin{equation}\label{gmrDens2}
  f(x) = \frac {r c^r}{\al^r \, \Ga(\al)} \,
    x^{-r - 1} \, \ga\left(\al + r, \frac {\al x} c\right).
\end{equation}
Since $\ga(\be, u) \to \Ga(\be)$ as $u \to \infty$, we have
\begin{equation*}
  f(x) \sim \frac {\Ga(\al + r)}{\al^r \, \Ga(\al)} \, \frac{r c^r}{x^{r + 1}}
\end{equation*}
as $x \to \infty$. In other words, the density of $X$ is asymptotically
proportional to the density of $M$ as $x \to \infty$.

For small values of $x$, note that $\ga(\be, u) \sim u^\be e^{-u}$ as $u \to 0$.
Hence, if we introduce the parameter $\la = \al/c$, then
\begin{equation*}
  f(x) \sim \frac r{\la^r \Ga(\al)} x^{-r - 1}(\la x)^{\al + r} e^{-\la x}
  = r \, \frac{\la^\al}{\Ga(\al)} \, x^{\al - 1} e^{-\la x}
\end{equation*}
as $x \to 0$. In other words, the density of $X$ is asymptotically proportional
to the density of $\Gam(\al, \al/c)$ as $x \to 0$. Since $\Gam (\al, \al/c)$ has
mean $c$, the parameter $c$ can be understood as the average score of players at
the lower end of the skill spectrum.

To understand $\al$, we look to the fact that $X \mid M \sim \Gam(\al, \al/M)$.
Given $M$, we can think of $X$ as being driven by $\al$ exponential clocks, each
with mean $M/\al$. Each clock represents a time to failure, and when all clocks
expire, the player has reached the termination event. Since $\al$ denotes the
number of such clocks, a higher value of $\al$ indicates that more failures are
needed to trigger the end of the game. We also have $\Var(X \mid M) = M^2/\al$.
Hence, $\al$ can also be understood through the fact that $1/\sqrt{\al}$ is the
coefficient of variation of $X$ given $M$.

For computational purposes, it may be more efficient to rewrite \eqref{gmrDens2}
in terms of the logarithm of the gamma function and the regularized lower
incomplete gamma function, $P(\be, u) = \ga(\be, u)/\Ga (\be)$. In this case, we
have
\begin{equation*}
  f(x) = r \left( \frac{c}{\al} \right)^r
    \exp(\log \Ga(\al + r) - \log \Ga(\al))
    x^{-r - 1} P\left( \al + r, \frac{\al x}{c} \right)
\end{equation*}
for $x > 0$.

\subsection{Video game leaderboards}\label{S:lboards}

For our final example, we consider a single-player video game in which an
individual player tries to score as many points as possible before the game
ends. If $X$ is a random score of a random player, then we will assume that $X
\sim \Gamer(r, c, \al)$, where $r = 7/3$, $c = 28$, and $\al = 3$. The values of
these parameters are arbitrarily chosen for the sake of the example. The values
of $r$ and $c$ give the distribution a mean of about 50 and a decay rate that
approximately matches the decay rate in the global Tetris leaderboard (see
\url{https://kirjava.xyz/tetris-leaderboard/}). The choice $\al = 3$ indicates a
game in which the player has 3 ``lives,'' which is a typical gaming structure,
especially in classic arcade video games. Finally, we assume that the actual
score displayed by the game is rounded to the nearest integer and capped at 499.

A group of 10 friends get together and play this game. Each friend plays the
game a different number of times. In this case, the agents are the players and
the actions are the scores they earn each time they play.

The 10 friends all have their own usernames that they use when playing the game.
The usernames are Asparagus Soda, Goat Radish, Potato Log, Pumpkins, Running
Stardust, Sweet Rolls, The Matrix, The Pianist Spider, The Thing, and Vertigo
Gal. We will consider three different scenarios for this example.

\subsubsection{Players with matching scores}

In our first scenario, the 10 friends generate the scores given in Table
\ref{Tb:scores1}. Note that in that table, the scores are listed in increasing
order. To get an overview of the data, we can place the 10 players in a
leaderboard, ranked by their high score, as shown in Table \ref{Tb:lboard1}.

\begin{table}
  \centering
  \begin{tabular}{|l|l|}
    \hline
    \multicolumn{1}{|c|}{\textbf{Username}} &
      \multicolumn{1}{|c|}{\textbf{Scores}} \\ \hline
    Pumpkins &            12, 21, 25, 25, 26, 27, 30, 33, 34, 34, 36,
                          42, 44, 44, 48, 55, 67, 69 \\ \hline
    Potato Log &          18, 21, 21, 22, 23, 25, 29, 29, 32, 33, 47,
                          53, 54, 56, 57, 65, 75 \\ \hline
    The Thing &           10, 16, 16, 19, 19, 25, 25, 26, 29, 32, 35,
                          37, 42, 44, 59, 60 \\ \hline
    Running Stardust &    23, 38, 62, 71, 138, 149, 151 \\ \hline
    Sweet Rolls &         15, 23, 56, 71, 98, 130 \\ \hline
    Vertigo Gal &         10, 30, 40, 56, 87, 92 \\ \hline
    Asparagus Soda &      17, 43, 55 \\ \hline
    The Matrix &          11, 15 \\ \hline
    Goat Radish &         38 \\ \hline
    The Pianist Spider &  3 \\ \hline
  \end{tabular}
  \caption{Player scores for Video Game Scenario 1}
  \label{Tb:scores1}
\end{table}

\begin{table}
  \centering
  \begin{tabular}{|c|l|c|c|c|c|}
    \hline
    \textbf{rank}       & \textbf{name}     & \textbf{hi score} &
    \textbf{avg score}  & \textbf{NDP avg}  & \textbf{\# games} \\ \hline
    1  & Running Stardust   & 151 & 90 & 80 & 7  \\ \hline
    2  & Sweet Rolls        & 130 & 66 & 55 & 6  \\ \hline
    3  & Vertigo Gal        & 92  & 52 & 52 & 6  \\ \hline
    4  & Potato Log         & 75  & 39 & 39 & 17 \\ \hline
    5  & Pumpkins           & 69  & 37 & 38 & 18 \\ \hline
    6  & The Thing          & 60  & 31 & 32 & 16 \\ \hline
    7  & Asparagus Soda     & 55  & 38 & 40 & 3  \\ \hline
    8  & Goat Radish        & 38  & 38 & 71 & 1  \\ \hline
    9  & The Pianist Spider & 32  & 32 & 37 & 1  \\ \hline
    10 & The Matrix         & 15  & 13 & 43 & 2  \\ \hline
  \end{tabular}
  \caption{Leaderboard for Video Game Scenario 1}
  \label{Tb:lboard1}
\end{table}

To model this scenario, we take $L = 500$, so that $S = \{0, 1, \ldots, 499\}$.
We take $\ka = \ep = 1$ and let
\[
  p_\l = \begin{cases}
    F(\l + 0.5) - F(\l - 0.5) & \text{if $0 \le \l < 499$},\\
    1 - F(498.5) &\text{if $\l = 499$},
  \end{cases}
\]
where $F$ is the distribution function of a $\Gamer(7/3, 28, 3)$ distribution.
For the data, we have $M = 10$, the number $N_m$ is the number of scores in the
$m$th row of Table \ref{Tb:scores1}, and $y_{mn}$ is the $n$th score in the
$m$th row. Note that since the model only depends on $y_{mn}$ through the row
counts $\ol y_{m\l}$, the order in which the scores are listed in the vector
$y_m$ is not relevant. In this scenario, we generated $K = 40000$ weighted
simulations, and obtained an effective sample size of about 326. In computing
the simulation weights as in \eqref{scaledWt}, we used a log scale factor of 42.

The long-term average score of the player in the $m$th row of Table 
\ref{Tb:scores1} will be $A(\th_m)$, where $A(x) = \sum_{\l = 0}^{499} \l x_\l$.
For example, using \eqref{agent}, the expected long-term average score of
Running Stardust is $E[A(\th_4) \mid Y = y] \approx 79.65$. These conditional
expectations, rounded to the nearest integer, are shown in the ``NDP avg''
column of Table \ref{Tb:lboard1}.

Looking at these averages, we can see at least two players whose numbers seem
unusual. The first is Goat Radish. They played only one game and scored a 38,
which is a relatively low score compared to the rest of the group. And yet the
NDP model has given them an expected long-term average score of 71. Not only is
this counterintuitive, it is also inconsistent with how the model treated The
Pianist Spider.

The reason for this behavior can be seen in Table \ref{Tb:scores1}. There is
only one other player that managed to score exactly 38 in one of their games:
Running Stardust. So from the model's perspective, there is a reasonable chance
that Goat Radish and Running Stardust have similar scoring tendencies. Since
Running Stardust happens to be the top player, this leads to an unusually high
long-term estimate for Goat Radish.

Our intuition is able to dismiss this line of reasoning because we know, for
instance, that there is very little difference between a score of 38 and 39. Had
Goat Radish scored a 39 instead, our predictions should not change that much.
But we only know this because we are viewing the positive real numbers as more
than just a set. We are viewing them as a totally ordered set with the Euclidean
metric. The NDP model is not designed to utilize these properties of the state
space. From its perspective, the number ``38'' is just a label. It is nothing
more than the name of a particular element of the state space, and it happens to
be an element that only two players were able to hit.

We see similar behavior in the model's forecast for The Matrix, who scored an 11
and a 15 in their two games. No one else scored an 11, but exactly one other
player managed to score exactly 15, and that was Sweet Rolls, who happens to be
the second best player. Just as with Goat Radish, this causes the model to
generate an unintuitively high value for The Matrix's long-term average score.

To test this explanation, we are led to our second scenario.

\subsubsection{Matching scores removed}

The scores in our second scenario are the same as in our first, but we changed
Goat Radish's 38 to a 39, and The Matrix's 15 to a 14. (See Table
\ref{Tb:scores2}.) The scores 14 and 39 are unique in that no other player
achieved exactly those scores. We reran the model, again generating $K = 40000$
weighted simulations. This time, we obtained an effective sample size of about
22.3. To save time, we deleted the two heaviest simulations, leaving $K = 39998$
simulations with an effective sample size of about 1099. The new expected
long-term averages are shown in Table \ref{Tb:lboard2}.

We now see that Goat Radish and The Matrix have lower, more reasonable long-term
averages according to the model. Likewise, Running Stardust and Sweet Rolls have
slightly higher averages. In the first scenario, their averages were brought
down because of their associations with Goat Radish and The Matrix.

\begin{table}
  \centering
  \begin{tabular}{|l|l|}
    \hline
    \multicolumn{1}{|c|}{\textbf{Username}} &
      \multicolumn{1}{|c|}{\textbf{Scores}} \\ \hline
    Pumpkins &            12, 21, 25, 25, 26, 27, 30, 33, 34, 34, 36,
                          42, 44, 44, 48, 55, 67, 69 \\ \hline
    Potato Log &          18, 21, 21, 22, 23, 25, 29, 29, 32, 33, 47,
                          53, 54, 56, 57, 65, 75 \\ \hline
    The Thing &           10, 16, 16, 19, 19, 25, 25, 26, 29, 32, 35,
                          37, 42, 44, 59, 60 \\ \hline
    Running Stardust &    23, 38, 62, 71, 138, 149, 151 \\ \hline
    Sweet Rolls &         15, 23, 56, 71, 98, 130 \\ \hline
    Vertigo Gal &         10, 30, 40, 56, 87, 92 \\ \hline
    Asparagus Soda &      17, 43, 55 \\ \hline
    The Matrix &          11, \textbf{14} \\ \hline
    Goat Radish &         \textbf{39} \\ \hline
    The Pianist Spider &  3 \\ \hline
  \end{tabular}
  \caption{Player scores for Video Game Scenario 2}
  \label{Tb:scores2}
\end{table}

\begin{table}
  \centering
  \begin{tabular}{|c|l|c|c|c|c|}
    \hline
    \textbf{rank}       & \textbf{name}     & \textbf{hi score} &
    \textbf{avg score}  & \textbf{NDP avg}  & \textbf{\# games} \\ \hline
    1  & Running Stardust   & 151 & 90 & 84 & 7  \\ \hline
    2  & Sweet Rolls        & 130 & 66 & 62 & 6  \\ \hline
    3  & Vertigo Gal        & 92  & 52 & 51 & 6  \\ \hline
    4  & Potato Log         & 75  & 39 & 39 & 17 \\ \hline
    5  & Pumpkins           & 69  & 37 & 38 & 18 \\ \hline
    6  & The Thing          & 60  & 31 & 31 & 16 \\ \hline
    7  & Asparagus Soda     & 55  & 38 & 39 & 3  \\ \hline
    8  & Goat Radish        & 38  & 38 & 43 & 1  \\ \hline
    9  & The Pianist Spider & 32  & 32 & 37 & 1  \\ \hline
    10 & The Matrix         & 15  & 13 & 28 & 2  \\ \hline
  \end{tabular}
  \caption{Leaderboard for Video Game Scenario 2}
  \label{Tb:lboard2}
\end{table}

\subsubsection{Players with only a few games}

In our third scenario, the ten friends generated the scores in Table
\ref{Tb:scores3}. We use the same $L$, $\ka$, $\ep$, $\bmp$, and $M$ as in the
first scenario. Also as it was there, the number $N_m$ is the number of scores
in the $m$th row of Table \ref{Tb:scores3}, and $y_{mn}$ is the $n$th score in
the $m$th row. Note, however, that the username in the $m$th row has changed in
the current scenario. We again used a log scale factor of 42 and generated $K =
40000$ weighted simulations, obtaining an effective sample size of about 39.
This time, we deleted the 26 heaviest simulations, leaving $K = 39974$
simulations and an effective sample size of about 207. As before, the resulting
long-term expected averages, $E[A(\th_m) \mid Y = y]$, are shown in Table
\ref{Tb:lboard3}.

In this example, we focus our attention on Asparagus Soda, who it situated at
No.~4 on the leaderboard, but played the game only once. The question is, does
he deserve to be at No.~4? Is he truly the fourth-best player among the ten
friends? For example, Potato Log, who is at No.~3, played the game 20 times and
only managed to get a high score of 87. Asparagus Soda almost matched that high
score in a single attempt. Intuitively, it seems clear that Asparagus Soda is
the better player and should rank higher than Potato Log.

It is less clear how Asparagus Soda compares to Pumpkins, the No.~2 player.
Neither of them made a lot of attempts, but Asparagus Soda has the higher
average score. Which one is more likely to have the higher long-term average
score? If they had a contest where they each played a single game and the higher
score wins, who should we bet on?

\begin{table}
  \centering
  \begin{tabular}{|l|p{5in}|}
    \hline
    \multicolumn{1}{|c|}{\textbf{Username}} &
      \multicolumn{1}{|c|}{\textbf{Scores}} \\ \hline
    Vertigo Gal &         45, 100, 118, 121, 125, 130, 133, 145, 161,
                          173, 173, 187, 190, 192, 193, 200, 220, 223,
                          256, 275, 314, 354, 388, 475, 524 \\ \hline
    Potato Log &          4, 13, 13, 16, 19, 19, 19, 19, 23, 24, 25,
                          26, 31, 38, 41, 43, 44, 47, 51, 87 \\ \hline
    The Thing &           4, 6, 9, 19, 25, 27, 28, 38, 39, 40 \\ \hline
    The Matrix &          13, 15, 17, 32, 32, 61, 78 \\ \hline
    Running Stardust &    21, 23, 51, 61, 65 \\ \hline
    Goat Radish &         23, 25, 34, 51 \\ \hline
    Pumpkins &            49, 65, 84, 117 \\ \hline
    Sweet Rolls &         26, 65 \\ \hline
    Asparagus Soda &      86 \\ \hline
    The Pianist Spider &  62 \\ \hline
  \end{tabular}
  \caption{Player scores for Video Game Scenario 3}
  \label{Tb:scores3}
\end{table}

\begin{table}
  \centering
  \begin{tabular}{|c|l|c|c|c|c|}
    \hline
    \textbf{rank}       & \textbf{name}     & \textbf{hi score} &
    \textbf{avg score}  & \textbf{NDP avg}  & \textbf{\# games} \\ \hline
    1  & Vertigo Gal        & 475 & 207 & 198 & 25 \\ \hline
    2  & Pumpkins           & 117 & 79  & 72  & 4  \\ \hline
    3  & Potato Log         & 87  & 30  & 31  & 20 \\ \hline
    4  & Asparagus Soda     & 86  & 86  & 67  & 1  \\ \hline
    5  & The Matrix         & 78  & 35  & 37  & 7  \\ \hline
    6  & Running Stardust   & 65  & 44  & 45  & 5  \\ \hline
    6  & Sweet Rolls        & 65  & 46  & 52  & 2  \\ \hline
    8  & The Pianist Spider & 62  & 62  & 56  & 1  \\ \hline
    9  & Goat Radish        & 51  & 33  & 34  & 4  \\ \hline
    10 & The Thing          & 40  & 24  & 26  & 10 \\ \hline
  \end{tabular}
  \caption{Leaderboard for Video Game Scenario 3}
  \label{Tb:lboard3}
\end{table}

\paragraph{Asparagus Soda vs.~Potato Log.}

Looking at Table \ref{Tb:scores3}, we see that Asparagus Soda corresponds to $m
= 9$ and Potato Log corresponds to $m = 2$. Table \ref{Tb:lboard3} shows us that
$E[A(\th_9) \mid Y = y] \approx 67$ and $E[A(\th_2) \mid Y = y] \approx 31$. In
other words, the NDP model gives Asparagus Soda a much higher expected long-term
average score than Potato Log. This confirms our intuition that Asparagus Soda
is the better player. But because Asparagus Soda played only one game, the model
should have a lot more uncertainty surrounding Asparagus Soda's forecasted mean.
To see this, we can compare approximate densities for $\cL(A(\th_9) \mid Y = y)$
and $\cL(A (\th_2) \mid Y = y)$. (See Figure \ref{F:aspvspot}.)

\begin{figure}
  \centering
  \begin{subfigure}[b]{0.45\linewidth}
    \includegraphics[width=\linewidth]{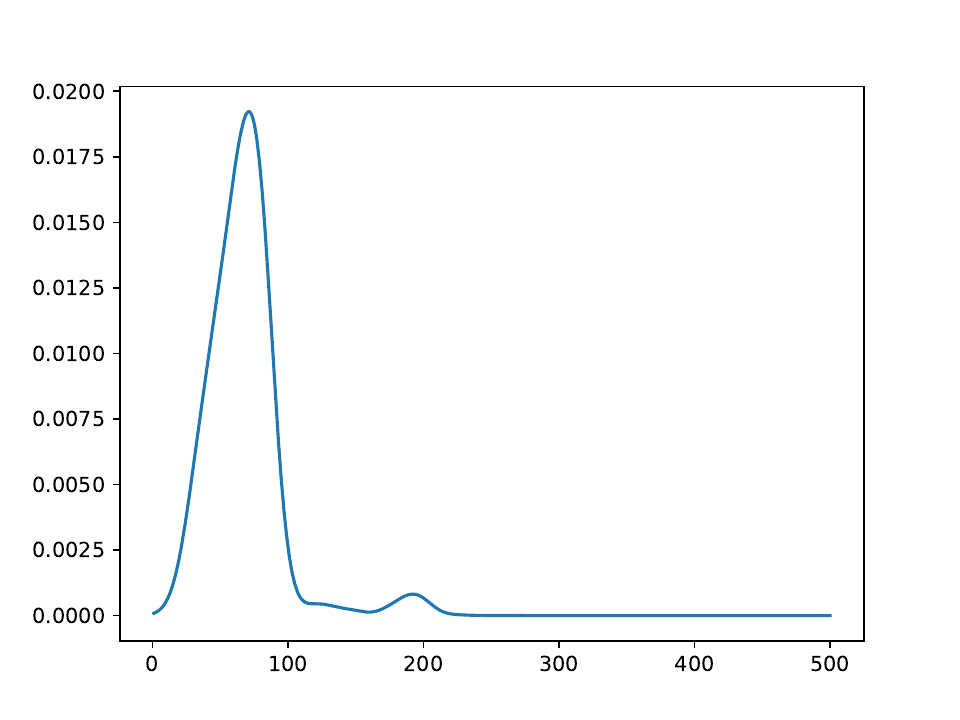}
    \caption{density for $\cL(A(\th_9) \mid Y = y)$}
  \end{subfigure}
  \begin{subfigure}[b]{0.45\linewidth}
    \includegraphics[width=\linewidth]{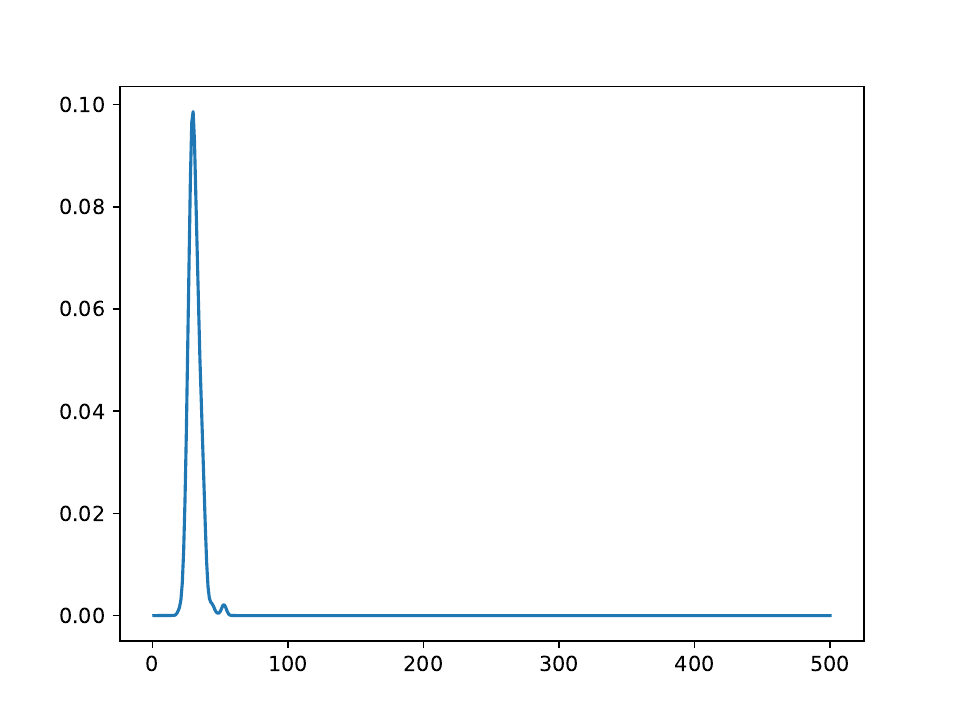}
    \caption{density for $\cL(A(\th_2) \mid Y = y)$}
  \end{subfigure}
  \begin{subfigure}[b]{0.45\linewidth}
    \includegraphics[width=\linewidth]{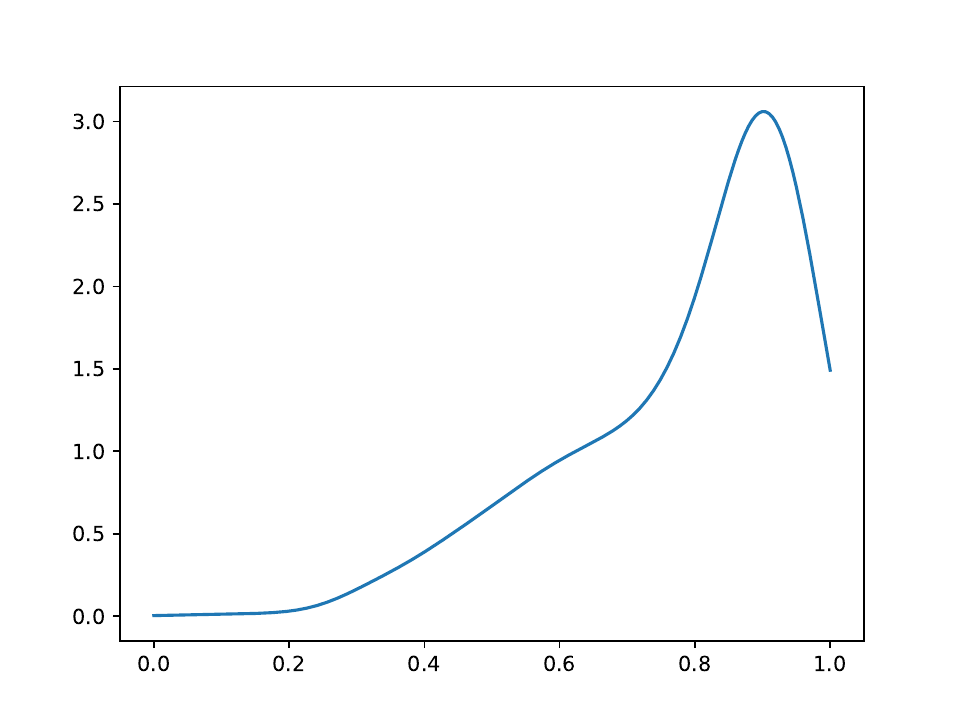}
    \caption{density for $\cL(C(\th_9, \th_2) \mid Y = y)$}
  \end{subfigure}
  \caption{Asparagus Soda ($m = 9$) vs.~Potato Log ($m = 2$)}
  \label{F:aspvspot}
\end{figure}

As is visually evident, Asparagus Soda's density is supported on a much wider
interval. In this way, the model acknowledges the possibility that Asparagus
Soda's actual long-term average score is lower than Potato Log's. The
probability that this is the case is $P(A(\th_9) < A(\th_2) \mid Y = y)$. If we
define $\Phi: \bR^{10 \times 500} \to \bR$ by $\Phi((x_{m\l})) = A(x_9) - A
(x_1)$, then we can use \eqref{agent} to obtain $P(A(\th_9) < A(\th_2) \mid Y =
y) \approx 0.049$. In other words, according to the model, there is a 95\%
chance that Asparagus Soda is a better player than Potato Log.

Now suppose the two of them had a contest in which they each played the game
once and the higher score wins. What is the probability that Asparagus Soda
would win this contest? If we define $C: \bR^{500} \times \bR^{500} \to \bR$ by
$C(x, y) = \sum_{\l > \l'} x_\l y_{\l'}$ and then $C(\th_9, \th_2)$ is the
(unknown) probability that Asparagus Soda beats Potato Log in this single-game
contest. The actual probability, given our observations $Y = y$, is then $E[C
(\th_9, \th_2) \mid Y = y]$, which, according to \eqref{agent}, is approximately
$0.786$. That is, Asparagus Soda has about a 79\% chance of beating Potato Log
in a contest involving a single play of the game. To visualize the uncertainty
around this probability, we can graph an approximate density for $\cL(C(\th_9,
\th_2) \mid Y = y)$. This is done in Figure \ref{F:aspvspot}(c). The graph shows
that although the conditional mean of $C(\th_9, \th_2)$ is about 79\%, the
conditional mode is much higher.

\paragraph{Asparagus Soda vs.~Pumpkins.}

We now turn our attention to comparing Asparagus Soda, who played only once, to
Pumpkins, who played four times. (See Figure \ref{F:aspvspum}.)

\begin{figure}
  \centering
  \begin{subfigure}[b]{0.45\linewidth}
    \includegraphics[width=\linewidth]{images/asparagus.pdf}
    \caption{density for $\cL(A(\th_9) \mid Y = y)$}
  \end{subfigure}
  \begin{subfigure}[b]{0.45\linewidth}
    \includegraphics[width=\linewidth]{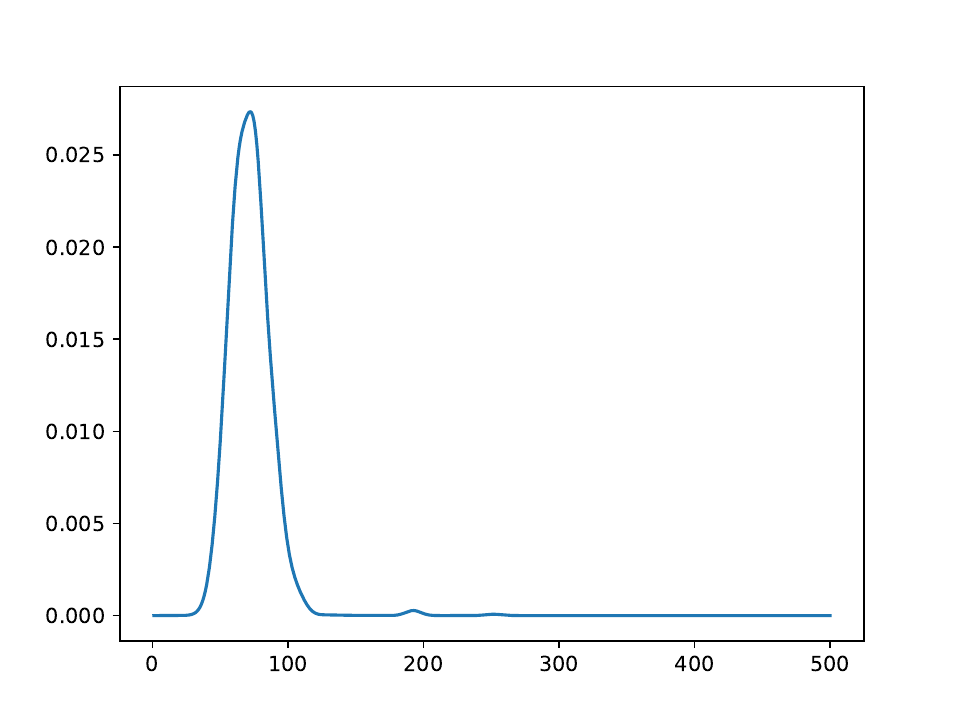}
    \caption{density for $\cL(A(\th_7) \mid Y = y)$}
  \end{subfigure}
  \begin{subfigure}[b]{0.45\linewidth}
    \includegraphics[width=\linewidth]{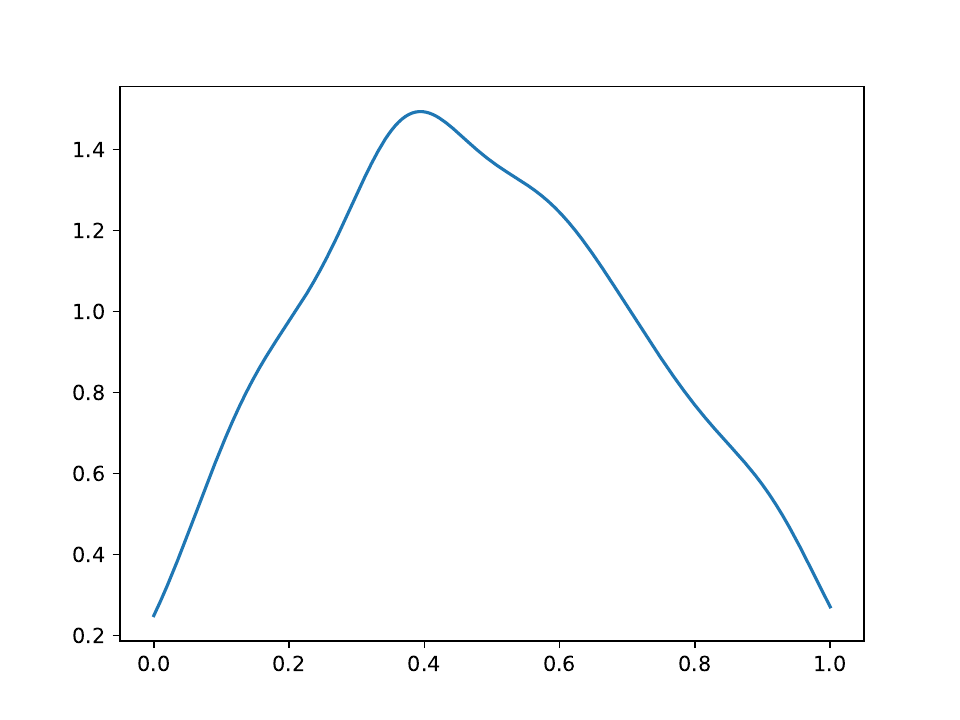}
    \caption{density for $\cL(C(\th_9, \th_7) \mid Y = y)$}
  \end{subfigure}
  \caption{Asparagus Soda ($m = 9$) vs.~Pumpkins ($m = 7$)}
  \label{F:aspvspum}
\end{figure}

Looking at Table \ref{Tb:scores3}, we see that Asparagus Soda corresponds to $m
= 9$ and Pumpkins corresponds to $m = 7$. Table \ref{Tb:lboard3} shows us that
$E[A(\th_9) \mid Y = y] \approx 67$ and $E[A(\th_7) \mid Y = y] \approx 72$. We
can visualize the model's uncertainty around Pumpkins' expected long-term
average by graphing an approximate density for $\cL(A(\th_7) \mid Y = y)$. This
is done in Figure \ref{F:aspvspum}(b).

Visually comparing this graph with the corresponding one for Asparagus Soda in
Figure \ref{F:aspvspum}(a), we see that the two long-term averages have
comparable degrees of uncertainty. Using \eqref{agent}, we have $P(A(\th_9) < A
(\th_2) \mid Y = y) \approx 0.625$, meaning there is a 62\% chance that Pumpkins
is the better player.

We can also consider a single-game contest between Asparagus Soda and Pumpkins.
As above, we can use \eqref{agent} to compute $E[C(\th_9, \th_7) \mid Y = y]
\approx 0.484$, meaning that Asparagus Soda has a 48\% chance of beating
Pumpkins in a single-game contest. To visualize the uncertainty around this
probability, we can graph an approximate density for $\cL(C(\th_9, \th_7) \mid Y
= y)$. (See Figure \ref{F:aspvspum}(c).)

\bibliographystyle{plain}
\bibliography{ndp-paper}

\begin{thebibliography}{10}

\bibitem{Antoniak1974}
Charles~E. Antoniak.
\newblock Mixtures of {D}irichlet processes with applications to {B}ayesian
  nonparametric problems.
\newblock {\em Ann. Statist.}, 2:1152--1174, 1974.

\bibitem{Barrientos2012}
Andr{\'e}s~F. Barrientos, Alejandro Jara, and Fernando~A. Quintana.
\newblock {On the Support of MacEachern’s Dependent Dirichlet Processes and
  Extensions}.
\newblock {\em Bayesian Analysis}, 7(2):277 -- 310, 2012.

\bibitem{Beckett1994}
Laurel Beckett and Persi Diaconis.
\newblock Spectral analysis for discrete longitudinal data.
\newblock {\em Adv. Math.}, 103(1):107--128, 1994.

\bibitem{Berry1979}
Donald~A. Berry and Ronald Christensen.
\newblock Empirical {B}ayes estimation of a binomial parameter via mixtures of
  {D}irichlet processes.
\newblock {\em Ann. Statist.}, 7(3):558--568, 1979.

\bibitem{Dunson2007}
David~B. Dunson, Natesh Pillai, and Ju-Hyun Park.
\newblock Bayesian density regression.
\newblock {\em Journal of the Royal Statistical Society Series B: Statistical
  Methodology}, 69(2):163--183, 03 2007.

\bibitem{Ferguson1973}
Thomas~S. Ferguson.
\newblock A {B}ayesian analysis of some nonparametric problems.
\newblock {\em Ann. Statist.}, 1:209--230, 1973.

\bibitem{Kemp2006}
Charles Kemp, Joshua~B. Tenenbaum, Thomas~L. Griffiths, Takeshi Yamada, and
  Naonori Ueda.
\newblock Learning systems of concepts with an infinite relational model.
\newblock In {\em Proceedings of the 21st National Conference on Artificial
  Intelligence - Volume 1}, AAAI'06, page 381–388, Boston, Massachusetts,
  2006. AAAI Press.

\bibitem{Kloek1978}
T.~Kloek and H.~K. van Dijk.
\newblock Bayesian estimates of equation system parameters: An application of
  integration by monte carlo.
\newblock {\em Econometrica}, 46(1):1--19, 1978.

\bibitem{Kong1992}
Augustine Kong.
\newblock A note on importance sampling using standardized weights.
\newblock Technical Report 348, Chicago, Illinois 60637, July 1992.

\bibitem{Kong1994}
Augustine Kong, Jun~S. Liu, and Wing~Hung Wong.
\newblock Sequential imputations and {B}ayesian missing data problems.
\newblock {\em J. Amer. Statist. Assoc.}, 89(425):278--288, March 1994.

\bibitem{Liu1996}
Jun~S. Liu.
\newblock Nonparametric hierarchical {B}ayes via sequential imputations.
\newblock {\em Ann. Statist.}, 24(3):911--930, 1996.

\bibitem{MacEachern1999}
S.~N. MacEachern.
\newblock Dependent nonparametric processes.
\newblock In {\em ASA Proceedings of the Section on Bayesian Statistical
  Science, Alexandria, VA}. American Statistical Association, 1999.

\bibitem{MacEachern2000}
S.~N. MacEachern.
\newblock Dependent {D}irichlet processes.
\newblock Technical report, 2000.

\bibitem{Mueller2004}
Peter Müller, Fernando Quintana, and Gary Rosner.
\newblock A method for combining inference across related nonparametric
  bayesian models.
\newblock {\em Journal of the Royal Statistical Society Series B: Statistical
  Methodology}, 66(3):735--749, 07 2004.

\bibitem{Orbanz2015}
P.~{Orbanz} and D.~M. {Roy}.
\newblock Bayesian models of graphs, arrays and other exchangeable random
  structures.
\newblock {\em IEEE Transactions on Pattern Analysis and Machine Intelligence},
  37(2):437--461, 2015.

\bibitem{Quintana2022}
Fernando~A. Quintana, Peter M{\"u}ller, Alejandro Jara, and Steven~N.
  MacEachern.
\newblock {The Dependent Dirichlet Process and Related Models}.
\newblock {\em Statistical Science}, 37(1):24 -- 41, 2022.

\bibitem{Rodriguez2008}
Abel Rodríguez, David~B Dunson, and Alan~E Gelfand.
\newblock The nested {D}irichlet process.
\newblock {\em Journal of the American Statistical Association},
  103(483):1131--1154, 2008.

\bibitem{Scott1992}
David~W. Scott.
\newblock {\em Multivariate density estimation}.
\newblock Wiley Series in Probability and Mathematical Statistics: Applied
  Probability and Statistics. John Wiley \& Sons, Inc., New York, 1992.
\newblock Theory, practice, and visualization, A Wiley-Interscience
  Publication.

\bibitem{Teh2006}
Yee~Whye Teh, Michael~I. Jordan, Matthew~J. Beal, and David~M. Blei.
\newblock Hierarchical {D}irichlet processes.
\newblock {\em J. Amer. Statist. Assoc.}, 101(476):1566--1581, 2006.

\bibitem{Xu2006}
Zhao Xu, Volker Tresp, Kai Yu, and Hans-Peter Kriegel.
\newblock Infinite hidden relational models.
\newblock In {\em Proceedings of the Twenty-Second Conference on Uncertainty in
  Artificial Intelligence}, UAI'06, page 544–551, Arlington, Virginia, USA,
  2006. AUAI Press.

\end{thebibliography}

\end{document}